\documentclass[12pt]{amsart}
\usepackage{graphics}
\usepackage{latexsym}
\usepackage{amsmath}
\usepackage{amssymb,amsthm,amsfonts}
\usepackage{lscape}
\usepackage{amscd}
\usepackage{syntonly}
\title[The jumping phenomenon]
{The jumping phenomenon of the dimensions of Bott-Chern cohomology groups and Aeppli cohomology groups}

\author[Jiezhu Lin]{Jiezhu Lin}
\author[Xuanming Ye]{Xuanming Ye}

\address{Department of Mathematics, Guangzhou University, No 230, Waihuan Road West, Guangzhou, 510006, P.R.China}
\address{Department of Mathematics, Sun Yat-Sen University, Guangzhou 510275, P.R.China}
\email{jlin@gzhu.edu.cn}
\email{yexm3@mail.sysu.edu.cn}

\begin{document}
\theoremstyle{plain}
\newtheorem{thm}{Theorem}[section]
\newtheorem{theorem}[thm]{Theorem}
\newtheorem{lemma}[thm]{Lemma}
\newtheorem{corollary}[thm]{Corollary}
\newtheorem{proposition}[thm]{Proposition}
\newtheorem{addendum}[thm]{Addendum}
\newtheorem{variant}[thm]{Variant}
\theoremstyle{definition}
\newtheorem{construction}[thm]{Construction}
\newtheorem{notations}[thm]{Notations}
\newtheorem{question}[thm]{Question}
\newtheorem{problem}[thm]{Problem}
\newtheorem{remark}[thm]{Remark}
\newtheorem{remarks}[thm]{Remarks}
\newtheorem{definition}[thm]{Definition}
\newtheorem{claim}[thm]{Claim}
\newtheorem{assumption}[thm]{Assumption}
\newtheorem{assumptions}[thm]{Assumptions}
\newtheorem{properties}[thm]{Properties}
\newtheorem{example}[thm]{Example}
\newtheorem{conjecture}[thm]{Conjecture}
\numberwithin{equation}{thm}

\newcommand{\sA}{{\mathcal A}}
\newcommand{\sB}{{\mathcal B}}
\newcommand{\sC}{{\mathcal C}}
\newcommand{\sD}{{\mathcal D}}
\newcommand{\sE}{{\mathcal E}}
\newcommand{\sF}{{\mathcal F}}
\newcommand{\sG}{{\mathcal G}}
\newcommand{\sH}{{\mathcal H}}
\newcommand{\sI}{{\mathcal I}}
\newcommand{\sJ}{{\mathcal J}}
\newcommand{\sK}{{\mathcal K}}
\newcommand{\sL}{{\mathcal L}}
\newcommand{\sM}{{\mathcal M}}
\newcommand{\sN}{{\mathcal N}}
\newcommand{\sO}{{\mathcal O}}
\newcommand{\sP}{{\mathcal P}}
\newcommand{\sQ}{{\mathcal Q}}
\newcommand{\sR}{{\mathcal R}}
\newcommand{\sS}{{\mathcal S}}
\newcommand{\sT}{{\mathcal T}}
\newcommand{\sU}{{\mathcal U}}
\newcommand{\sV}{{\mathcal V}}
\newcommand{\sW}{{\mathcal W}}
\newcommand{\sX}{{\mathcal X}}
\newcommand{\sY}{{\mathcal Y}}
\newcommand{\sZ}{{\mathcal Z}}
\newcommand{\A}{{\mathbb A}}
\newcommand{\B}{{\mathbb B}}
\newcommand{\C}{{\mathbb C}}
\newcommand{\D}{{\mathbb D}}
\newcommand{\E}{{\mathbb E}}
\newcommand{\F}{{\mathbb F}}
\newcommand{\G}{{\mathbb G}}
\newcommand{\HH}{{\mathbb H}}
\newcommand{\I}{{\mathbb I}}
\newcommand{\J}{{\mathbb J}}
\renewcommand{\L}{{\mathbb L}}
\newcommand{\M}{{\mathbb M}}
\newcommand{\N}{{\mathbb N}}
\renewcommand{\P}{{\mathbb P}}
\newcommand{\Q}{{\mathbb Q}}
\newcommand{\R}{{\mathbb R}}
\newcommand{\SSS}{{\mathbb S}}
\newcommand{\T}{{\mathbb T}}
\newcommand{\U}{{\mathbb U}}
\newcommand{\V}{{\mathbb V}}
\newcommand{\W}{{\mathbb W}}
\newcommand{\X}{{\mathbb X}}
\newcommand{\Y}{{\mathbb Y}}
\newcommand{\Z}{{\mathbb Z}}
\newcommand{\id}{{\rm id}}
\newcommand{\rank}{{\rm rank}}
\newcommand{\END}{{\mathbb E}{\rm nd}}
\newcommand{\End}{{\rm End}}
\newcommand{\Eis}{{\rm Eis}}
\newcommand{\Hg}{{\rm Hg}}
\newcommand{\tr}{{\rm tr}}
\newcommand{\Sl}{{\rm Sl}}
\newcommand{\Gl}{{\rm Gl}}
\newcommand{\Gr}{{\rm Gr}}
\newcommand{\Cor}{{\rm Cor}}
\newcommand{\dcechg}{\delta\!\!\!\check{\delta}}
\newcommand{\del}{\partial}
\newcommand{\delbar}{\overline{\del}}
\newcommand{\tempo}{\mathbf{t}}
\newcommand\SO{{\rm SO}}
\newcommand\SU{{\rm SU}}
\newcommand\fra{{\mathfrak a}} 
\newcommand\frg{{\mathfrak g}}
\newcommand\frh{{\mathfrak h}}
\newcommand\frk{{\mathfrak K}}
\newcommand\frn{{\mathfrak n}}
\newcommand\frs{{\mathfrak s}}
\newcommand\frt{{\mathfrak t}}
\newcommand\gc{\frg_\mathbb{C}}
\newcommand\Real{{\mathfrak R}{\frak e}\,} 
\newcommand\Imag{{\mathfrak I}{\frak m}\,}
\newcommand\nilm{\Gamma\backslash G}
\newcommand\Mod{{\cal M}}
\newcommand\db{{\bar{\partial}}}
\newcommand\zzz{{\!\!\!}}
\newcommand{\la}{\langle}
\newcommand{\sla}{\slash\!\!\!}
\newcommand{\opiccolo}[1]{\mathrm{o}\left(\left|#1\right|\right)}
\newcommand{\ra}{\rangle}
\newcommand\sqi{{\sqrt{-1\,}}}

\maketitle

\footnotetext[1]{This work was supported by National Natural Science Foundation of China (Grant No. 11201491 and 11201090), Doctoral Fund of Ministry of Education of China
 (Grant No. 20124410120001 and 20120171120009 ) and the Foundation of Research Funds for Young Teachers Training Project
 (Grant No. 34000-3161248).}

\begin{abstract}\ \rm Let $X$ be a compact complex
manifold, and let  $\pi: \mathcal{X} \rightarrow B$ be a small
deformation of $X$, the dimensions of the Bott-Chern cohomology
groups $H_{\rm BC}^{p,q}(X(t))$ and Aeppli cohomology groups $H_{\rm A}^{p,q}(X(t))$ may vary under this deformation.
In this paper, we will study the deformation obstructions of a
$(p,q)$ class in the central fiber $X$. In particular, we obtain
an explicit formula for the obstructions and apply this formula to
the study of small deformations of the Iwasawa manifold.
\\
\end{abstract}

\section{Introduction}
\renewcommand{\theequation}
{1.\arabic{equation}} \setcounter{equation}{-1}
Let $X$ be a compact complex manifold and
$\pi:\mathcal{X}\rightarrow B$ be a family of complex manifolds
such that $\pi^{-1}(0)=X$. Let $X_{t}=\pi^{-1}(t)$ denote the
fibre of $\pi$ over the point $t\in B$. In \cite{ye}, the author has studied the jumping phenomenon of
hodge numbers $h^{p,q}$ of $X$ by studying the deformation obstructions of a
$(p,q)$ class in the central fiber $X$. In particular, the author obtained
an explicit formula for the obstructions and apply this formula to
the study of small deformations of the Iwasawa manifold. Besides the Hodge numbers, the dimensions of
Bott-Chern cohomology groups and the dimensions of
Aeppli cohomology groups are also important invariant of complex structures. In \cite{angella-1}, D. Angella
has studied the small deformations of the Iwasawa manifold and found that  the dimensions of
Bott-Chern cohomology groups and the dimensions of
Aeppli cohomology groups are not deformation invariants.

 In this paper, we
will study the Bott-Chern cohomology and Aeppli cohomology by studying
the hypercohomology of a complex $\sB_{p,q}^{\bullet}$ constructed in \cite{Schw}. In \cite{Schw}, M. Schweitzer proved that $$H^{p,q}_{\rm BC}(X)\cong\HH^{p+q}(X,\sB_{p,q}^{\bullet}),$$
and
$$H^{p,q}_{\rm A}(X)\cong\HH^{p+q+1}(X,\sB_{p+1,q+1}^{\bullet}).$$
As the author did in \cite{ye}, we will such study the jumping phenomenons from the viewpoint of obstruction
theory. More precisely, for a certain small deformation
$\mathcal{X}$ of $X$ parameterized by a basis $B$ and a certain
class $[\theta]$ of the hypercohomology group
$\HH^{l}(X,\sB_{p,q}^{\bullet})$, we will try to find out the obstruction
to extend it to an element of the relative hypercohomology group
group $\HH^{l}(\sX,\sB_{p,q;\sX/B}^{\bullet})$. We will
call those elements which have non trivial obstruction the
obstructed elements. And then we will see that these elements will
play an important role when we study the jumping phenomenon. Because we will see that the existence of the
obstructed elements is a sufficient condition for
the variation of the dimensions of Bott-Chern cohomology and Aeppli cohomology.

In $\S2$ we will summarize the results of
M. Schweitzer about Bott-Chern cohomology and Aeppli cohomology, from which we can define
the relative Bott-Chern cohomology and Aeppli cohomology on $X_n$ where $X_n$ is the $n$th order
deformation of $\pi:\mathcal{X}\rightarrow B$. We will also introduce some important maps which will
be used in the calculation of the obstructions in $\S4$. In $\S3$,
we will try to explain why we
need to consider the obstructed elements. The relation between the
jumping phenomenon of Bott-Chern cohomology and Aeppli cohomology  and the obstructed elements is
the following.
\begin{theorem}
 \label{main011}
Let $\pi: \mathcal{X}\to B$ be a small deformation of the central
fibre compact complex manifold $X$. Now we consider $\dim \HH^l(X(t),\sB_{p,q;t}^{\bullet})$  as a function
 of $t\in B$. It jumps at $t=0$ if there exists an element
 $[\theta]$ either in $\HH^l(X,\sB_{p,q}^{\bullet})$  or in $\HH^{l-1}(X,
\sB_{p,q}^{\bullet}) $ and a minimal natural number $n\geq 1$ such that
the n-th order obstruction
 $$ o_{n}([\theta])\neq 0.$$

\end{theorem}
In $\S4$ we will get a formula for the obstruction to the
extension we mentioned above.
\begin{theorem} \label{main12}
Let $\pi:\mathcal{X}\rightarrow B$ be a deformation of
$\pi^{-1}(0)=X$, where $X$ is a compact complex manifold. Let
$\pi_{n}:X_{n}\rightarrow B_{n}$ be the $n$th order deformation of
$X$. For arbitrary $[\theta]$ belongs to $\HH^l(X,\sB_{p,q}^{\bullet})$,
suppose we can extend $[\theta]$ to order $n-1$ in
$\HH^l(X_{n-1},\sB_{p,q;X_{n-1}/B_{n-1}}^{\bullet})$. Denote such element by
$[\theta_{n-1}]$. The obstruction of the extension of $[\theta]$
to $n$th order is given by:
$$ o_{n}([\theta])=-\partial^{\bar{\partial},\sB}_{X_{n-1}/B_{n-1}} \circ \kappa_{n} \llcorner \circ \partial^{\sB,\bar{\partial}}_{X_{n-1}/B_{n-1}}([\theta_{n-1}])- \bar{\partial}^{{\partial},\sB}_{X_{n-1}/B_{n-1}} \circ \bar{\kappa}_{n} \llcorner \circ \bar{\partial}^{\sB,{\partial}}_{X_{n-1}/B_{n-1}}([\theta_{n-1}]),$$
where $\kappa_{n}$ is the $n$th order Kodaira-Spencer class and $\bar{\kappa}_{n}$  is
the $n$th order Kodaira-Spencer class of the deformation $\bar{\pi}:\bar{\sX} \rightarrow \bar{B}$.
$\partial^{\bar{\partial},\sB}_{X_{n-1}/B_{n-1}}$, $\bar{\partial}^{{\partial},\sB}_{X_{n-1}/B_{n-1}}$ , $\partial^{\sB,\bar{\partial}}_{X_{n-1}/B_{n-1}}$ and $\bar{\partial}^{\sB,{\partial}}_{X_{n-1}/B_{n-1}}$
are the  maps defined in $\S2£¤$.
\end{theorem}
In $\S5$ we will use this formula to study carefully
the example given by Iku Nakamura and D. Angella, i.e. the small deformation of
the
Iwasama manifold and discuss some phenomenons.

\section{The Relative Bott-Chern Cohomology and Aeppli Cohomology of $X_n$ and the Representation of their Cohomology Classes}\

\subsection{The Bott-Chern(Aeppli) Cohomology and Bott-Chern(Aeppli) Hypercohomology } \ \ \\
All the details of this subsection can be found in \cite{Schw}.
Let $X$ be a compact complex manifold. The Dolbeault cohomology groups $H^{p,q}_{\db}(X)$, and more generally the terms $E_r^{p,q}(X)$ in the
Fr\"olicher spectral sequence~\cite{Fro}, are well-known finite dimensional invariants of the complex manifold $X$.
On the other hand, the
Bott-Chern and Aeppli cohomologies define additional complex invariants of $X$ given, respectively, by~\cite{Ae,BC}
$$
H^{p,q}_{\rm BC}(X)=
{\ker\{d\colon \sA^{p,q}(X)\longrightarrow \sA^{p+q+1}(X) \}
\over
{\rm im}\,\{\partial\db\colon \sA^{p-1,q-1}(X)\longrightarrow \sA^{p,q}(X) \}},
$$
and
$$
H^{p,q}_{\rm A}(X)=
{\ker\{\partial\db\colon \sA^{p,q}(X)\longrightarrow \sA^{p+1,q+1}(X) \}
\over
{\rm im}\,\{\partial\colon \sA^{p-1,q}(X)\longrightarrow \sA^{p,q}(X) \}
+
{\rm im}\,\{\db\colon \sA^{p,q-1}(X)\longrightarrow \sA^{p,q}(X) \}}.
$$
By the Hodge theory developed in \cite{Schw}, all these complex invariants are also finite dimensional and one has the isomorphisms
$H^{p,q}_{\mathrm{A}}(X)\cong H^{n-q,n-p}_{\mathrm{\rm BC}}(X)$.
Notice that $H^{q,p}_{\mathrm{\rm BC}}(X) \cong H^{p,q}_{\mathrm{\rm BC}}(X)$ by complex conjugation.
For any $r\geq 1$ and for any $p,q$, there are natural maps
$$
H^{p,q}_{\mathrm{\rm BC}}(X) \longrightarrow E_r^{p,q}(X)
\quad\quad\quad
{\mbox{\rm and }}
\quad\quad\quad
E_r^{p,q}(X) \longrightarrow H^{p,q}_{\mathrm{A}}(X).
$$
Recall that $E_1^{p,q}(X)\cong H^{p,q}_{\db}(X)$ and that the terms for $r=\infty$ provide a decomposition of the de Rham cohomology of the manifold, i.e.
$H^k_{\rm dR}(X,\mathbb{C})\cong \oplus_{p+q=k} E_{\infty}^{p,q}(X)$.
From now on we shall denote by $h^{p,q}_{\mathrm{\rm BC}}(X)$
the dimension of the cohomology group $H^{p,q}_{\mathrm{\rm BC}}(X)$.
The Hodge numbers will be denoted simply by $h^{p,q}(X)$ and the Betti numbers by $b_{k}(X)$.
For any given $p\geq 1,q\geq 1$, we define the complex of sheaves $\sL^{\bullet}_{p,q}$  by
$$\sL^k_{p,q}=\bigoplus_{\substack{r+s=k \\ r<p,s<q}}\sA^{r,s}\quad \mbox{if }k\leq p+q-2,$$
$$\sL^{k-1}_{p-1,q-1}=\bigoplus_{\substack{r+s=k \\ r\geq p,s\geq q}}\sA^{r,s}\quad \mbox{if }k\geq p+q,$$
and the differential:
$$\sL_{p,q}^0 \stackrel{pr_{\sL^1_{p,q}}\circ d}{\to} \sL^1_{p,q}\stackrel{pr_{\sL^2_{p,q}}\circ d}{\to} \ldots \to\sL^{k-2}_{p,q}\stackrel{\partial \bar{\partial}}{\to}\sL^{k-1}_{p,q}\stackrel{d}{\to}\sL^k_{p,q}\stackrel{d}{\to}\ldots$$
Then by the above construction, we have the following isomorphisms
$$H^{p,q}_{\rm BC}(X)=H^{p+q-1}(\sL_{p,q}^{\bullet}(X))\cong \HH^{p+q-1}(X,\sL_{p,q}^{\bullet}),$$
$$H^{p,q}_{\rm A}(X)=H^{p+q}(\sL_{p+1,q+1}^{\bullet}(X))\cong \HH^{p+q}(X,\sL_{p+1,q+1}^{\bullet}),$$
because $\sL^k_{p,q}$ are soft.

   We define a sub complex $\sS_{p,q}^{\bullet}$ of $\sL_{p,q}^{\bullet}$ by :
$$({\sS'}_p^{\bullet},\partial ):\; \sO \to \Omega^1 \to \ldots\to\Omega^{p-1}\to 0,\qquad ({\sS''}_q^{\bullet},\bar{\partial}):\; \bar{\sO}\to\bar{\Omega}^{1}\to\ldots\to\bar{\Omega}^{q-1}\to 0,$$
$$\sS_{p,q}^{\bullet}={\sS'}_p^{\bullet}+{\sS''}_q^{\bullet}:\; \sO+\bar{\sO} \to\Omega^1\oplus\bar{\Omega}^{1}\to\ldots\Omega^{p-1}\oplus\bar{\Omega}^{p-1}\to\bar{\Omega}^{p}\to\ldots\to\bar{\Omega}^{q-1}\to 0.$$
Note that the inclusion $\sS^{\bullet}\subset\sL^{\bullet}$ is an quasi-isomorphism \cite{Schw}.
There is another complex $\sB^{\bullet}_{p,q}$ used in \cite{Schw} which is defined by:

$$ \sB_{p,q}^{\bullet} :\; \C \stackrel{(+,-)}{\to}\sO \oplus\bar{\sO} \to\Omega^1\oplus\bar{\Omega}^{1}\to\ldots\Omega^{p-1}\oplus\bar{\Omega}^{p-1}\to\bar{\Omega}^{p}\to\ldots\to\bar{\Omega}^{q-1}\to 0.$$
and the following morphism of from $\sB^{\bullet}_{p,q}$ to $\sS^{\bullet}_{p,q}[1]$ is a quasi-isomorphism \cite{Schw}:
$$\begin{array}{ccccccc}
\C&\stackrel{(+,-)}{\to}&\sO\oplus\bar{\sO}&\to&\Omega^1\oplus\bar{\Omega}^{1}&\to&\ldots\\
\downarrow&&\downarrow +&&\downarrow&&\\
0&\to&\sO +\bar{\sO}&\to&\Omega^1\oplus\bar{\Omega}^{1}&\to&\ldots.
\end{array}$$


Therefore we have:
$$H^{p,q}_{\rm BC}(X)\cong\HH^{p+q}(X,\sL_{p,q}^{\bullet}[1])\cong\HH^{p+q}(X,\sS_{p,q}^{\bullet}[1])\cong\HH^{p+q}(X,\sB_{p,q}^{\bullet}),$$
and
$$H^{p,q}_{\rm A}(X)\cong\HH^{p+q}(X,\sL_{p+1,q+1}^{\bullet})\cong\HH^{p+q}(X,\sS_{p+1,q+1}^{\bullet})\cong\HH^{p+q+1}(X,\sB_{p+1,q+1}^{\bullet}),$$

\subsection{The Relative Bott-Chern Cohomology and Aeppli Cohomology of $X_n$ }
 \ \ \\
Let $\pi:\mathcal{X}\rightarrow B$ be a deformation of
$\pi^{-1}(0)=X$, where $X$ is a compact complex manifold. For
every integer $n\geq 0$, denote by
$B_{n}=Spec\,\mathcal{O}_{B,0}/m_{0}^{n+1}$ the $n$th order
infinitesimal neighborhood of the closed point $0\in B$ of the
base $B$. Let $X_{n}\subset \mathcal{X}$ be the complex space over
$B_{n}$. Let $\pi_{n}:X_{n}\rightarrow B_{n}$ be the $n$th order
deformation of $X$. Denote $\pi^{*}(m_{0})$ by $\mathcal{M}_{0}$. If we take the complex conjugation, we have
another complex structure of the differential manifold of $\mathcal{X}$, we denote this manifold by $\bar{\mathcal{X}}$ and $\pi$ induce a deformation $\bar{\pi}:\bar{\mathcal{X}}\rightarrow \bar{B}$  of $\bar{X}$.
Then we have $\bar{X}_{n}$ and $\bar{\pi}_{n}:\bar{X}_{n}\rightarrow \bar{B}_{n}.$ Let $\sC^{\omega}_B$ be the sheaf of
$\C$-valued real analytic functions on $B$. Denote $\sO_\sX^{\omega}= \pi^{*}(\sC^{\omega}_B)$, $\bar{\sO}_\sX^{\omega}= \bar{\pi}^{*}(\sC^{\omega}_B).
$ Let $m^{\omega}_{0}$ be the maximal idea of $\sC^{\omega}_{B,0}$ and
$\mathcal{M}_{0}^{\omega}= \pi^{*}(m^{\omega}_{0}) $, $\bar{\mathcal{M}}_{0}^{\omega}= \bar{\pi}^{*}(m^{\omega}_{0}) $ . For any sheaf of $\sO_\sX$(resp. $\bar{\sO}_\sX$) module $\sF$. Denote $\sF^{\omega}=\sF \otimes_{\sO_\sX} \sO_\sX^{\omega} $ (resp. $\sF^{\omega}=\sF \otimes_{\bar{\sO}_\sX} \bar{\sO}_\sX^{\omega}).$ Let ${\sO}_{X_n}^{\omega}= \sO_{\sX,0}^{\omega}/ ({\mathcal{M}}_{0}^{\omega})^{n+1}$£¬ ${\sO}_{\bar{X}_n}^{\omega}= \bar{\sO}_{\sX,0}^{\omega}/ ({\bar{\mathcal{M}}}_{0}^{\omega})^{n+1}$. For any sheaf of $\sO_{X_n}$(resp. ${\sO}_{\bar{X}_n}$) module $\sF$. Denote $\sF^{\omega}=\sF \otimes_{\sO_{X_n}} {\sO}_{X_n}^{\omega}$ (resp. $\sF^{\omega}=\sF \otimes_{\sO_{\bar{X}_n}} {\sO}_{\bar{X}_n}^{\omega}).$

For any given $p \geq 1, q \geq 1$, We define the complex $\sS^{\bullet}_{X_n/B_n}=\sS_{p,q;X_n/B_n}^{\bullet}
$ by:
$$({\sS'}_{p;X_n/B_n}^{\bullet},\partial_{X_n/B_n} ):\; \sO^{\omega}_{X_n} \to \Omega_{X_n/B_n}^{1;\omega} \to \ldots\to\Omega_{X_n/B_n}^{p-1;\omega}\to 0,$$
$$ ({\sS''}_{q;X_n/B_n}^{\bullet},\bar{\partial}_{X_n/B_n}):\; \sO^{\omega}_{\bar{X}_n} \to\bar{\Omega}_{X_n/B_n}^{1;\omega}\to\ldots\to\bar{\Omega}_{X_n/B_n}^{q-1;\omega}\to 0,$$
$\sS_{p,q;X_n/B_n}^{\bullet}={\sS'}_{p;X_n/B_n}^{\bullet}+{\sS''}_{q;X_n/B_n}^{\bullet}:\; $
 $$\sO^{\omega}_{X_n}+{\sO}_{\bar{X}_n}^{\omega} \to\Omega_{X_n/B_n}^{1;^{\omega}}\oplus\bar{\Omega}_{X_n/B_n}^{1;\omega}\to\ldots\Omega_{X_n/B_n}^{p-1;\omega}\oplus\bar{\Omega}_{X_n/B_n}^{p-1;\omega}\to\bar{\Omega}_{X_n/B_n}^{p;\omega}\to\ldots\to\bar{\Omega}_{X_n/B_n}^{q-1;\omega}\to 0.$$
We can also define $\sB^{\bullet}_{p,q;X_n/B_n}$  by:

$$ \sB_{p,q;X_n/B_n}^{\bullet} :\; \C^{\omega}_{B_n} \stackrel{(+,-)}{\to} \sO^{\omega}_{X_n} \oplus \sO^{\omega}_{\bar{X}_n} \to\Omega_{X_n/B_n}^{1;\omega}\oplus\bar{\Omega}_{X_n/B_n}^{1;\omega}\to\ldots\Omega_{X_n/B_n}^{p-1;\omega}\oplus\bar{\Omega}_{X_n/B_n}^{p-1;\omega} $$
$$\to\bar{\Omega}_{X_n/B_n}^{p;\omega}\to\ldots\to\bar{\Omega}_{X_n/B_n}^{q-1;\omega}\to 0,$$
where $\C^{\omega}_{B_n} = \pi^{-1} (\sC^{\omega}_{B,0}/(m^{\omega}_{0})^{n+1}).$

Then the Relative Bott-Chern cohomology and Aeppli cohomology of $X_n$ is defined by

$$H^{p,q}_{\rm BC}(X_n/B_n) \cong \HH^{p+q}(X,\sS_{p,q;X_n/B_n}^{\bullet}[1])\cong\HH^{p+q}(X_n,\sB_{p,q;X_n/B_n}^{\bullet}),$$
and
$$H^{p,q}_{\rm A}(X_n/B_n) \cong \HH^{p+q}(X,\sS_{p+1,q+1;X_n/B_n}^{\bullet})\cong\HH^{p+q+1}(X_n,\sB_{p+1,q+1;X_n/B_n}^{\bullet}).$$
\subsection{The Representation of the Relative Bott-Chern Cohomology and Aeppli Cohomology Classes }
 \ \ \\

In this subsection we will follow \cite{Schw} to construct a hypercocycle in $\check{Z}^{p+q}(X,\sB _{p,q}^{\bullet})$ to represent the relative Bott-Chern cohomology classes.  Let $[\theta]$ be an element of $H^{p,q}_{\rm BC}(X)$, represented by a closed $(p,q)$-form $\theta$. It is defined in $\HH^{p+q}(X,\sL_{p,q}[1]^{\bullet})$ by a hypercocycle, still denoted by  $\theta$ and defined by $\theta^{p,q}=\theta|_{U_j}$ and $\theta^{r,s}=0$ otherwise.
For gievn $p\geq 1$ and $q\geq 1$, there exists a hypercocycle $w=(c;u^{r,0};v^{0,s})\in\check{Z}^{p+q}(X,\sB_{p,q}^{\bullet})$ and an  hypercochain $\alpha=(\alpha^{r,s})\in\check{C}^{p+q-1}(X,\sL_{p,q}[1]^{\bullet})$ such that $\theta=\dcechg\alpha+w$. We represent the data in the following table:
$$\theta\longleftrightarrow\left[
\begin{array}{c|ccc}
\theta_v^{0,q-1}&&&\\
\vdots&&\alpha^{r,s}&\\
\theta_v^{0,0}&&&\\\hline
\theta_c&\theta_u^{0,0}&\cdots&\theta_u^{p-1,0}\end{array}\right]$$
The equality $\theta=\check{\delta}\alpha+w$ corresponds to the following relations:
$$(\bigstar)\left\{\begin{array}{rcll}
\theta^{p,q}&=&\bar{\partial} \alpha^{p-1,q-1}&\\
(-1)^{r+s}\check{\delta}\alpha^{r,s}&=&\bar{\partial}\alpha^{r,s-1}+\partial\alpha^{r-1,s}&\forall\, 1\leq r\leq p-1,\, 1\leq s\leq q-1\\
(-1)^s\check{\delta}\alpha^{0,s}&=&\bar{\partial}\alpha^{0,s-1}+\theta_v^{0,s}&\forall\, 1\leq s\leq q-1\\
(-1)^r\check{\delta}\alpha^{r,0}&=&\theta_u^{r,0}+\partial\alpha^{r-1,0}&\forall\, 1\leq r\leq p-1\\
\check{\delta}\alpha^{0,0}&=&\theta_u^{0,0}+\theta_v^{0,0}&\\
\check{\delta} \theta_u^{0,0}&=&\theta_c&
\end{array}\right.$$
Note that these relations involve relations of the hypercocycles for $\theta_u$ and $\theta_v$:
$$(-1)^r\check{\delta} \theta_u^{r,0}=\partial \theta_u^{r-1,0}\;\forall 1\leq r\leq p-1,\qquad (-1)^s\check{\delta} \theta_v^{0,s}=\bar{\partial} \theta_v^{0,s-1}\;\forall 1\leq s\leq q-1.$$

If $q=0$, we simply have:
$$\theta\longleftrightarrow\left(\theta_c,\theta_u^{0,0},\ldots,\theta_u^{p-1,0}\right)$$
with the relations:
$$\theta^{p,0}=\partial \theta_u^{p-1,0},\quad (-1)^r\check{\delta} \theta_u^{r,0}=\partial \theta_u^{r-1,0}\;\forall 1\leq r\leq p-1,\quad \check{\delta} \theta_u^{0,0}=\theta_c.$$
Similarly, if $p=0$, we have
$$\theta\longleftrightarrow\left(\theta_c,\theta_v^{0,0},\ldots,\theta_v^{0,q-1}\right)$$
with the relations:
$$\theta^{0,q}=-\bar{\partial} \theta_v^{0,q-1},\quad (-1)^s\check{\delta} \theta_v^{0,s}=\bar{\partial} \theta_v^{0,s-1}\;\forall 1\leq r\leq p-1,\quad -\check{\delta} \theta_v^{0,0}=\theta_c.$$

Similarly, let $[\theta]$ be an element of $H^{p,q}_{\rm BC}(X_n/B_n)$, then it can be represented by a \v Cech
hypercocycle $\theta_u$, $\theta_v$ and $\theta_c$ of $\check{Z}^{p+q}(X,\sB _{p,q;X_n/B_n}^{\bullet})$ with the relations:$$(-1)^r\check{\delta} \theta_u^{r,0}=\partial \theta_u^{r-1,0}\;\forall 1\leq r\leq p-1,\qquad (-1)^s\check{\delta} \theta_v^{0,s}=\bar{\partial} \theta_v^{0,s-1}\;\forall 1\leq s\leq q-1.$$$$  \check{\delta}\theta_u^{0,0}=\theta_c,\qquad -\check{\delta} \theta_v^{0,0}=\theta_c; $$
while for an element $[\theta]$ of $H^{p,q}_{\rm A}(X_n/B_n)$, it can be represented by a \v Cech
hypercocycle $\theta_u$ and $\theta_v$ of $\check{Z}^{p+q+1}(X,\sB _{p+1,q+1;X_n/B_n}^{\bullet})$ with the relations:$$(-1)^r\check{\delta} \theta_u^{r,0}=\partial \theta_u^{r-1,0}\;\forall 1\leq r\leq p,\qquad (-1)^s\check{\delta} \theta_v^{0,s}=\bar{\partial} \theta_v^{0,s-1}\;\forall 1\leq s\leq q,$$ $$  \check{\delta}\theta_u^{0,0}=\theta_c,\qquad -\check{\delta}\theta_v^{0,0}=\theta_c. $$

Before the end of this subsections, we will introduce some important maps which will be used in
the computation in $\S4$. \\
Define
 $$\partial^{\bar{\partial},\sB}_{X_n/B_n}: H^{\bullet}(X_n, \Omega^{p-1;\omega}_{X_{n}/B_{n}}) \rightarrow   \HH^{\bullet+p} (X_n,\sB_{p,q;X_n/B_n}^{\bullet}) $$ in the following way:\\
 Let $[\theta]$ be an element of $H^{\bullet}(X_n, \Omega^{p-1;\omega}_{X_{n}/B_{n}})$ then $\theta$ can be represented
 by a cocycle of $\check{Z}^{\bullet}(X,\Omega^{p-1;\omega}_{X_{n}/B_{n}})$, we define $\partial^{\bar{\partial},\sB}_{X_n/B_n}([\theta])$ to be the cohomology class associated to the hypercocyle in $\check{Z}^{p+\bullet}(X,\sB _{p,q;X_n/B_n}^{\bullet})$ given by $\theta_u^{p-1,0}=\theta$, $\theta_u^{r,0}=0\; \forall 0\leq r\leq p-2$ and $\theta_v^{0,r}=0\; \forall 0\leq r \leq q-1$, $ \theta_c=0. $
 \begin{lemma} \label{lemma1}
 $\partial^{\bar{\partial},\sB}_{X_n/B_n}$ is well defined.
 \end{lemma}
 \begin{proof}
 It is easy to check that the hypercochian given by $\theta_u$, $\theta_v$ and $\theta_c$ is a hypercocycle.
 On the other hand if there exists a cochain $\alpha^{'}$ in $\check{C}^{\bullet-1}(X,\Omega^{p-1;\omega}_{X_{n}/B_{n}})$ such that
 $\check{\delta} \alpha^{'}= \theta$, then if we take a hypercochian $\alpha$ in $\check{C}^{p+\bullet-1}(X,\sB _{p,q;X_n/B_n}^{\bullet})$ given by $\alpha_u^{p-1,0}=(-1)^{p-1}\alpha^{'}$, $\alpha_u^{r,0}=0\; \forall 0\leq r\leq p-2$ , $\alpha_v^{0,r}=0\; \forall 0\leq r \leq q-1$ and $\alpha_c=0$ we have $\dcechg\alpha = \partial^{\bar{\partial},\sB}_{X_n/B_n}([\theta])$. Therefore $\partial^{\bar{\partial},\sB}_{X_n/B_n}([\theta])=0.$
 \end{proof}
 Similarly, we can define $$\bar{\partial}^{{\partial},\sB}_{X_n/B_n}: H^{\bullet}(\bar{X}_n, \bar{\Omega}^{q-1;\omega}_{X_{n}/B_{n}}) \rightarrow   \HH^{\bullet+q} (X_n,\sB_{p,q;X_n/B_n}^{\bullet}) $$ in the following way:\\
 Let $[\theta]$ be an element of $H^{\bullet}(\bar{X}_n, \bar{\Omega}^{q-1;\omega}_{X_{n}/B_{n}})$ then $\theta$ can be represented
 by a cocycle of $\check{Z}^{\bullet}(\bar{X},\bar{\Omega}^{q-1;\omega}_{X_{n}/B_{n}})$, we define $\bar{\partial}^{{\partial},\sB}_{X_n/B_n}([\theta])$ to be the cohomology class associated to the hypercocyle in $\check{Z}^{q+\bullet}(X,\sB _{p,q;X_n/B_n}^{\bullet})$ given by $\theta_v^{0,q-1}=\theta$, $\theta_v^{0,r}=0\; \forall 0\leq r\leq q-2$ , $\theta_u^{r,0}=0\; \forall 0\leq r \leq p-1$ and $\theta_c=0$. This map is also well defined and the proof is just as lemma \ref{lemma1}.

 Define
 $$\partial^{\sB,\bar{\partial}}_{X_n/B_n}: \HH^{\bullet+p} (X_n,\sB_{p,q;X_n/B_n}^{\bullet}) \rightarrow  H^{\bullet}(X_n, \Omega^{p;\omega}_{X_{n}/B_{n}}) $$ in the following way:\\
 Let $[\theta]$ be an element of $ \HH^{\bullet+p} (X_n,\sB_{p,q;X_n/B_n}^{\bullet})$ then $\theta$ can be represented by a hypercocycle of $\check{Z}^{p+\bullet}(X,\sB _{p,q;X_n/B_n}^{\bullet})$, we define ${\partial}^{\sB,\bar{\partial}}_{X_n/B_n}([\theta])$ to be the cohomology class associated to the cocyle in $\check{Z}^{\bullet}(X,\Omega^{p;\omega}_{X_{n}/B_{n}})$ given by $\partial_{X_n/B_n} \theta_u^{p-1,0}$.
  \begin{lemma} \label{lemma2}
 $\partial^{\sB, \bar{\partial}}_{X_n/B_n}$ is well defined.
 \end{lemma}
  \begin{proof}
 At first we need to check the cochian given by $\partial_{X_n/B_n} \theta_u^{p-1,0}$ is a cocycle. In fact,
 since $\theta$ is a hypercocycle in $\check{Z}^{p+\bullet}(X,\sB _{p,q;X_n/B_n}^{\bullet})$, we have
 $(-1)^{p-1}\check{\delta} \theta_u^{p-1,0}= \partial_{X_n/B_n} \theta_u^{p-2,0}$, therefore, $\check{\delta} \partial_{X_n/B_n} \theta_u^{p-1,0}=(-1)^p\partial_{X_n/B_n}\circ \partial_{X_n/B_n}\theta_u^{p-2,0}=0$.

 On the other hand if there exists a cochain $\alpha$ in $\check{C}^{p+\bullet-1}(X,\sB _{p,q;X_n/B_n}^{\bullet})$ such that
 $\dcechg\alpha= \theta$, then if we take a cochian $\alpha^{'}$ in $\check{C}^{\bullet-1}(X,\Omega^{p;\omega}_{X_{n}/B_{n}})$ given by $\alpha^{'} = (-1)^p \partial_{X_n/B_n}  \alpha_u^{p-1,0}$, we have $ \check{\delta} \alpha^{'} = (-1)^p\check{\delta} \partial_{X_n/B_n}  \alpha_u^{p-1,0} = (-1)^{p+1} \partial_{X_n/B_n} \check{\delta} \alpha_u^{p-1,0} = (-1)^{p+1+p-1}\partial_{X_n/B_n}  \theta_u^{p-1,0}=\partial^{\sB,\bar{\partial},}_{X_n/B_n}([\theta])$. Therefore $\partial^{\bar{\partial},\sB}_{X_n/B_n}([\theta])=0.$
 \end{proof}
 Similarly, we can define $$\bar{\partial}^{\sB,{\partial}}_{X_n/B_n}: \HH^{\bullet+q} (X_n,\sB_{p,q;X_n/B_n}^{\bullet}) \rightarrow  H^{\bullet}(\bar{X}_n, \Omega^{q;\omega}_{X_{n}/B_{n}}) $$ in the following way:\\
 Let $[\theta]$ be an element of $ \HH^{\bullet+q} (X_n,\sB_{p,q;X_n/B_n}^{\bullet})$ then $\theta$ can be represented by a hypercocycle of $\check{Z}^{q+\bullet}(X,\sB _{p,q;X_n/B_n}^{\bullet})$, we define $\bar{\partial}^{\sB,{\partial}}_{X_n/B_n}([\theta])$ to be the cohomology class associated to the cocyle in $\check{Z}^{\bullet}(X,\bar{\Omega}^{q;\omega}_{X_{n}/B_{n}})$ given by $\bar{\partial}_{X_n/B_n} \theta_u^{0,q-1}$. This map is also well defined and the proof is just as lemma \ref{lemma2}.\\
 \begin{remark}
 The natural maps from $H^{p,q}_{\rm BC}(X_n/B_n)$ to  $H^{q}(X_n, \Omega^{p;\omega}_{X_{n}/B_{n}})$ and from $H^{q}(X_n, \Omega^{p;\omega}_{X_{n}/B_{n}})$ to $H^{p,q}_{\rm A}(X_n/B_n)$ mentioned in $\S2.1$ respectively are exactly the map:
$$\partial^{\sB,\bar{\partial}}_{X_n/B_n}: \HH^{q+p} (X_n,\sB_{p,q;X_n/B_n}^{\bullet})( \cong H^{p,q}_{\rm BC}(X_n/B_n)) \rightarrow  H^{q}(X_n, \Omega^{p;\omega}_{X_{n}/B_{n}}), $$
 $$\partial^{\bar{\partial},\sB}_{X_n/B_n}: H^{q}(X_n, \Omega^{p;\omega}_{X_{n}/B_{n}}) \rightarrow \HH^{q+p+1} (X_n,\sB_{p+1,q+1;X_n/B_n}^{\bullet})( \cong H^{p,q}_{\rm A}(X_n/B_n)) . $$ and we denote these maps by $r_{BC,\bar{\partial}} $ and $r_{\bar{\partial},A} .$ \\
 We also denote the following two maps:
 $$\partial^{\bar{\partial},\sB}_{X_n/B_n}: H^{q}(X_n, \Omega^{p-1;\omega}_{X_{n}/B_{n}}) \rightarrow \HH^{q+p} (X_n,\sB_{p,q;X_n/B_n}^{\bullet})( \cong H^{p,q}_{\rm BC}(X_n/B_n)), $$
 $$\partial^{\sB,\bar{\partial}}_{X_n/B_n}: \HH^{q+p+1} (X_n,\sB_{p+1,q+1;X_n/B_n}^{\bullet})( \cong H^{p,q}_{\rm A}(X_n/B_n)) \rightarrow  H^{q}(X_n, \Omega^{p+1;\omega}_{X_{n}/B_{n}}). $$
 by
$\partial^{\bar{\partial},BC}_{X_n/B_n}$ and $\partial^{A,\bar{\partial}}_{X_n/B_n}.$
 \end{remark}
The following lemma is an important observation which will be used for the computation in $\S4$.
\begin{lemma} \label{lemma4}
 Let $[\theta]$ be an element of $ \HH^{l} (X_n,\sB_{p,q;X_n/B_n}^{\bullet})$ which is represented by
 an element $\theta$ in $\check{Z}^{l}(X,\sB _{p,q;X_n/B_n}^{\bullet})$ given by $\theta_u$,
$\theta_v$ and $\theta_c$, then $\del_{X_n/B_n}(\theta- \theta_u^{p-1,0})$ is a hypercoboundary.
\end{lemma}
\begin{proof}
The hypercochian $\del_{X_n/B_n}(\theta- \theta_u^{p-1,0})$ is given by $(\del_{X_n/B_n}(\theta- \theta_u^{p-1,0}))_u^{r,0}=\del_{X_n/B_n} \theta_u^{r-1,0}, \forall 0 < r <p-1$ , $ (\del_{X_n/B_n}(\theta- \theta_u^{p-1,0}))_u^{0,0}=0$ , $(\del_{X_n/B_n}(\theta- \theta_u^{p-1,0}))_v^{0,r}=0, \forall 0 \leq r \leq q-1 $ and $(\del_{X_n/B_n}(\theta- \theta_u^{p-1,0}))_c=0$. Let $\alpha$ be the hypercochian in $\check{C}^{l}(X,\sB _{p,q;X_n/B_n}^{\bullet})$ given by $\alpha_u^{2r,0}=0, \forall 0 \leq r < p/2, $ $\alpha_u^{2r-1,0}=\theta_u^{2r-1,0}, \forall 0 < r \leq p/2$, $\alpha_v^{0,r}=0, \forall 0 \leq r \leq q-1 $ and $\alpha_c=0$  and
it is easy to see that $\dcechg \alpha =\del_{X_n/B_n}(\theta- \theta_u^{p-1,0})$. Therefore $\del_{X_n/B_n}(\theta- \theta_u^{p-1,0})$ is a hypercoboundary.
\end{proof}
Similarly, we have
\begin{lemma} \label{lemma5}
 Let $[\theta]$ be an element of $ \HH^{l} (X_n,\sB_{p,q;X_n/B_n}^{\bullet})$ which is represented by
 an element $\theta$ in $\check{Z}^{l}(X,\sB _{p,q;X_n/B_n}^{\bullet})$ given by $\theta_u$,
$\theta_v$ and $\theta_c$, then $\delbar_{X_n/B_n}(\theta- \theta_v^{0,q-1})$ is a hypercoboundary.
\end{lemma}



\section{The Jumping Phenomenon and Obstructions}\label{section3}
\renewcommand{\theequation}
{2.\arabic{equation}} \setcounter{equation}{-1}

There is a Hodge theory also for Bott-Chern and Aeppli cohomologies, see \cite{Schw}. More precisely, fixed a Hermitian metric on $X$, one has that
$$ H^{\bullet,\bullet}_{\rm BC}(X) \;\simeq\; \ker\tilde\Delta_{\rm BC} \qquad \text{ and }\qquad H^{\bullet,\bullet}_{\rm A}(X) \;\simeq\; \ker\tilde\Delta_{\rm A} \;,$$
where
$$ \tilde\Delta_{\rm BC} \;:=\;
\left(\del\delbar\right)\left(\del\delbar\right)^*+\left(\del\delbar\right)^*\left(\del\delbar\right)+\left(\delbar^*\del\right)\left(\delbar^*\del\right)^*+\left(\delbar^*\del\right)^*\left(\delbar^*\del\right)+\delbar^*\delbar+\del^*\del $$
and
$$ \tilde\Delta_{\rm A} \;:=\; \del\del^*+\delbar\delbar^*+\left(\del\delbar\right)^*\left(\del\delbar\right)+\left(\del\delbar\right)\left(\del\delbar\right)^*+\left(\delbar\del^*\right)^*\left(\delbar\del^*\right)+\left(\delbar\del^*\right)\left(\delbar\del^*\right)^* $$
are $4$-th order elliptic self-adjoint differential operators. In particular, one gets that
$$ \dim_\C H^{\bullet,\bullet}_{\sharp}(X) \;<\;+\infty \qquad \text{ for } \sharp\in\left\{\delbar,\,\del,\,BC,\,A\right\} \;.$$

Let $\pi:\mathcal{X}\rightarrow B$ be a deformation of
$\pi^{-1}(0)=X$, where $X$ is a compact complex manifold and $B$ is a neighborhood of the origin in $\C$. Note that $h^{p,q}_{\rm BC}(X(t))$ and $h^{p,q}_{\rm A}(X(t))$ are semi-continuous functions of $t \in B$ where $X(t)= \pi^{-1}(t)$ \cite{Schw}. Denote
the $ \tilde\Delta_{\rm BC}$ operator and the $ \tilde\Delta_{\rm A}$ on $X(t)$ by $ \tilde\Delta_{BC,t}$ and $ \tilde\Delta_{A,t}$. From the prove of the semi-continuity of $h^{p,q}_{\rm BC}(X(t))$ ($h^{p,q}_{\rm A}(X(t))$) on \cite{Schw}, we can see that $h^{p,q}_{\rm BC}(X(t))$ ( $h^{p,q}_{\rm A}(X(t))$ )does not jump at the point $t=0$ if and only if all the $ \tilde\Delta_{BC,0}$($ \tilde\Delta_{A,0}$)-harmonic forms on $X$ can be extended
to relative $ \tilde\Delta_{BC,t}$($ \tilde\Delta_{A,t}$)-harmonic forms on a neighborhood of $0 \in B$ which is real analytic in the direction of $B$, since the $ \tilde\Delta_{BC,t}$($ \tilde\Delta_{A,t}$) varies real analytic on $B$.
The above condition is equivalent to the following: all the cohomology classes $[\theta]$ in $H^{p,q}_{\rm BC}(X)$($H^{p,q}_{\rm A}(X)$) can be extended to a relative $d_t-closed$( $\del_t\delbar_t-closed$) forms $\theta(t)$  such that $[\theta(t)]\neq 0$ on a neighborhood of $0 \in B$ which is real analytic on the direction of $B$.
Therefore in order to study the jumping phenomenon, we need to study the extension obstructions. So we need to study the obstructions of the extension of the cohomology classes in $\HH^{\bullet}(X,\sB_{p,q}^{\bullet})$ to a relative cohomology classes in $\HH^{\bullet}(X_n,\sB_{p,q;X_n/B_n}^{\bullet})$. 
 We denote the following complex 
 $$ \pi^{-1}(m_0^{\omega}/(m_0^{\omega})^{n+1}) \stackrel{(+,-)}{\to} \mathcal{M}_{0}^{\omega} / (\mathcal{M}^{\omega} _{0})^{n+1}\otimes \sO^{\omega}_{X_n} \oplus \bar{\mathcal{M}}_{0}^{\omega} / (\bar{\mathcal{M}}^{\omega} _{0})^{n+1}\otimes \sO^{\omega}_{\bar{X}_n} $$$$\to \mathcal{M}_{0}^{\omega} / (\mathcal{M}^{\omega} _{0})^{n+1}\otimes \Omega_{X_n/B_n}^{1;\omega}\oplus \bar{\mathcal{M}}_{0}^{\omega} / (\bar{\mathcal{M}}^{\omega} _{0})^{n+1}\otimes \bar{\Omega}_{X_n/B_n}^{1;\omega}$$$$\to\ldots \mathcal{M}_{0}^{\omega} / (\mathcal{M}^{\omega} _{0})^{n+1}\otimes\Omega_{X_n/B_n}^{p-1;\omega}\oplus \bar{\mathcal{M}}_{0}^{\omega} / (\bar{\mathcal{M}}^{\omega} _{0})^{n+1}\otimes \bar{\Omega}_{X_n/B_n}^{p-1;\omega}$$
$$\to \bar{\mathcal{M}}_{0}^{\omega} / (\bar{\mathcal{M}}^{\omega} _{0})^{n+1}\otimes \bar{\Omega}_{X_n/B_n}^{p;\omega}\to\ldots\to \bar{\mathcal{M}}_{0}^{\omega} / (\bar{\mathcal{M}}^{\omega} _{0})^{n+1}\otimes\bar{\Omega}_{X_n/B_n}^{q-1;\omega}\to 0,$$
 
 by $\mathcal{M}_{0}^{\omega} / (\mathcal{M}^{\omega} _{0})^{n+1}\otimes \sB_{p,q;X_{n-1}/B_{n-1}}^{\bullet}$

 Now we consider the following exact sequences:
$$ 0 \rightarrow \mathcal{M}_{0} / \mathcal{M}^{n+1}_{0}\otimes \sB_{p,q;X_{n-1}/B_{n-1}}^{\bullet}
\rightarrow \sB_{p,q;X_{n}/B_{n}}^{\bullet} \rightarrow
\sB_{p,q;X_{0}/B_{0}}^{\bullet}\rightarrow 0$$ which induces a long
exact
sequence\\
$$ 0 \rightarrow \HH^0(X_{n},\mathcal{M}_{0} / \mathcal{M}^{n+1}_{0}\otimes
\sB_{p,q;X_{n-1}/B_{n-1}}^{\bullet})\rightarrow \HH^0(X_{n},
\sB_{p,q;X_{n}/B_{n}}^{\bullet})\rightarrow
\HH^0(X,\sB_{p,q;X_{0}/B_{0}}^{\bullet})
$$
$$ \rightarrow \HH^1(X_{n},\mathcal{M}_{0} / \mathcal{M}^{n+1}_{0}\otimes
\sB_{p,q;X_{n-1}/B_{n-1}}^{\bullet}) \rightarrow ... .$$ Let $[\theta]$ be a cohomology class in $\HH^l(X,\sB_{p,q;X_{0}/B_{0}}^{\bullet})$.
The obstruction for the extension of
$[\theta]$ to a relative cohomology classes in $\HH^l(X_{n},\sB_{p,q;X_{n}/B_{n}}^{\bullet})$
comes from the non trivial image of the
connecting homomorphism
$\delta^{*}:\HH^l(X,\sB_{p,q;X_{0}/B_{0}}^{\bullet})\rightarrow
\HH^{l+1}(X_{n},\mathcal{M}_{0} / \mathcal{M}^{n+1}_{0}\otimes
\sB_{p,q;X_{n-1}/B_{n-1}}^{\bullet})$. We denote this obstruction by $o_n([\theta])$.
On the other hand, for a given real direction $\frac{\del}{\del x}$ on $B$, if there exits $n \in \N$, such that $o_i([\theta])=0, \forall i \leq n$
and $o_n([\theta]) \neq 0. $ Let $\theta_{n-1}$ be a $n-1$ th order extension of $\theta$ to a relative cohomology classes in $\HH^l(X_{n-1},\sB_{p,q;X_{n-1}/B_{n-1}}^{\bullet})$. $\dcechg\theta_{n-1}=0$ up to order $n-1$.
Now if we consider $\dcechg\theta_{n-1}/x^n$, it is easy to check that $\dcechg\theta_{n-1}/x^n$ is an extension
of a non trivial  cohomology classes $[\dcechg\theta_{n-1}/x^n(0)]$ in $\HH^l(X,\sB_{p,q})$ while $[\dcechg\theta_{n-1}/x^n(x_0)]$ is trivial  in $X(x_0)$ as a cohomology classes in $\HH^l(X(x_0),\sB_{p,q;x_0})$ if $x_0 \neq 0$. From the above discussion, we have the following theorem:
\begin{theorem}
 \label{main0}
Let $\pi: \mathcal{X}\to B$ be a small deformation of the central
fibre compact complex manifold $X$. Now we consider $\dim \HH^l(X(t),\sB_{p,q;t}^{\bullet})$  as a function
 of $t\in B$. It jumps at $t=0$ if there exists an element
 $[\theta]$ either in $\HH^l(X,\sB_{p,q}^{\bullet})$  or in $\HH^{l-1}(X,
\sB_{p,q}) $ and a minimal natural number $n\geq 1$ such that
the n-th order obstruction
 $$ o_{n}([\theta])\neq 0.$$

\end{theorem}

\section{The Formula for the Obstructions}\label{section4}
\renewcommand{\theequation}
{3.\arabic{equation}} \setcounter{equation}{-1}
Since these obstructions we discussed in the previous section are so important when we consider the
problem of jumping phenomenon of Bott-Chern cohomology and Aeppli cohomology, we try to find out an
explicit calculation for such obstructions in this section. As we had done in \cite{ye}, we need some preparation.
 Cover $X$ by open sets $U_{i}$ such that, for
arbitrary $i$, $U_{i}$ is small enough. More precisely,
$U_{i}$ is stein and the following exact sequence splits\\
$$ 0 \rightarrow \pi_{n}^{*}(\Omega_{B_{n}})^{\omega}(U_{i})
\rightarrow \Omega_{X_{n}} ^{\omega}(U_{i})\rightarrow
\Omega_{X_{n}/B_{n}}^{\omega}(U_{i})\rightarrow 0;$$
$$ 0 \rightarrow \bar{\pi}_{n}^{*}(\Omega_{\bar{B}_{n}})^{\omega}(U_{i})
\rightarrow \Omega_{\bar{X}_{n}}^{\omega} (U_{i})\rightarrow
\bar{\Omega}_{X_{n}/B_{n}}^{\omega}(U_{i})\rightarrow 0.$$ So we have a map
$\varphi_{i}:  \Omega_{X_{n}/B_{n}}^{\omega} (U_{i}) \oplus \bar{\Omega}_{X_{n}/B_{n}} ^{\omega}(U_{i})\rightarrow
\Omega_{X_{n}}^{\omega} (U_{i})\oplus \Omega_{\bar{X}_{n}}^{\omega}(U_{i})$, such that $\varphi_{i}|_{\Omega_{X_{n}/B_{n}} ^{\omega} (U_{i})}(
\Omega_{X_{n}/B_{n}}^{\omega} (U_{i}))\oplus
\pi_{n}^{*}(\Omega_{B_{n}})^{\omega}(U_{i}) \cong \Omega_{X_{n}} ^{\omega}(U_{i}) $ and $\varphi_{i}|_{\bar{\Omega}_{X_{n}/B_{n}}^{\omega} (U_{i})}(
\bar{\Omega}_{X_{n}/B_{n}} ^{\omega}(U_{i}))\oplus
\bar{\pi}_{n}^{*}(\Omega_{\bar{B}_{n}})^{\omega}(U_{i}) \cong \Omega_{\bar{X}_{n}} ^{\omega}(U_{i}) $.
This decomposition determines a local decomposition of the
exterior differentiation $\partial_{X_{n}}$($\delbar_{X_{n}}$) in
$\Omega^{\bullet;\omega}_{X_{n}}$ (resp. $\bar{\Omega}^{\bullet;\omega}_{X_{n}}$)on each $U_{i}$\\
$$ \partial_{X_{n}}=\del^{i}_{B_{n}}+\del^{i}_{X_{n}/B_{n}} (resp. \del_{\bar{X}_{n}}=\delbar^{i}_{{B}_{n}}+\delbar^{i}_{X_{n}/B_{n}}).$$
By definition, $\del_{X_{n}/B_{n}}$ and $\delbar_{X_{n}/B_{n}}$ are given by $ \varphi_{i}^{-1}
\circ \del^{i}_{X_{n}/B_{n}}\circ\varphi_{i}$ and $ \varphi_{i}^{-1}
\circ \delbar^{i}_{X_{n}/B_{n}}\circ\varphi_{i}$. \\
\indent Denote the set of alternating $q$-cochains $\beta$ with
values in $\mathcal{F}$ by
$\check{C}^{q}(\mathbf{U},\mathcal{F})$, i.e. to each
$q+1$-tuple, $i_{0}<i_{1}...< i_{q}$, $\beta$ assigns a section
$\beta(i_{0},i_{1},..., i_{q})$ of $\mathcal{F}$ over
$U_{i_{0}}\cap U_{i_{1}} \cap ... \cap U_{i_{q}}$. \vskip 0.1cm
Let us still using $\varphi_{i}$ denote the following map,
\begin{eqnarray*}
\varphi_{i}: \pi_{n}^{*}(\Omega_{B_{n}})^{\omega} \wedge
\Omega^{p;\omega}_{X_{n}/B_{n}}(U_{i}) \oplus \bar{\pi}_{n}^{*}(\Omega_{\bar{B}_{n}})^{\omega} \wedge \bar{\Omega}^{p;\omega}_{X_{n}/B_{n}}(U_{i} ) & \rightarrow &
\Omega^{p+1;\omega}_{X_{n}}(U_{i}) \oplus \Omega^{p+1;\omega}_{\bar{X}_{n}}(U_{i})\\
\varphi_{i}(\omega_{1}\wedge\beta_{i_{1}}\wedge...\wedge\beta_{i_{p}}+\omega_{2}\wedge{\beta^{'}}_{j_{1}}\wedge...\wedge{\beta^{'}}_{j_{p}})
 & = &
\omega_{1}\wedge\varphi_{i}(\beta_{i_{1}})\wedge...\wedge\varphi_{i}(\beta_{i_{p}}) \\
&&+\omega_{2}\wedge\varphi_{i}({\beta^{'}}_{j_{1}})\wedge...\wedge\varphi_{i}({\beta^{'}}_{j_{p}}).
\end{eqnarray*}

Define
$\varphi:\check{C}^{q}(\mathbf{U},\pi_{n}^{*}(\Omega_{B_{n}})^{\omega} \wedge
\Omega^{p;\omega}_{X_{n}/B_{n}} \oplus \bar{\pi}_{n}^{*}(\Omega_{\bar{B}_{n}})^{\omega} \wedge \bar{\Omega}^{p;\omega}_{X_{n}/B_{n}})
\rightarrow \check{C}^{q}(\mathbf{U},\Omega^{p+1;\omega}_{{X}_{n}} \oplus \Omega^{p+1;\omega}_{\bar{X}_{n}})$ by\\
$$ \varphi(\beta)(i_{0},i_{1},..., i_{q})=\varphi_{i_{0}}(\beta(i_{0},i_{1},...,
i_{q})) \qquad \forall \beta \in
\check{C}^{q}(\mathbf{U},\pi_{n}^{*}(\Omega_{B_{n}})^{\omega} \wedge
\Omega^{p;\omega}_{X_{n}/B_{n}} \oplus \bar{\pi}_{n}^{*}(\Omega_{\bar{B}_{n}})^{\omega} \wedge \bar{\Omega}^{p;\omega}_{X_{n}/B_{n}}),$$ where $i_{0}<i_{1}...< i_{q}$.\vskip 0.1cm Define the
total Lie derivative with respect to $B_{n}$\\
$$ L_{B_{n}}: \check{C}^{q}(\mathbf{U},\Omega^{p;\omega}_{X_{n}}\oplus {\Omega}^{p;\omega}_{\bar{X}_{n}})
\rightarrow \check{C}^{q}(\mathbf{U},\Omega^{p+1;\omega}_{X_{n}}\oplus {\Omega}^{p+1;\omega}_{\bar{X}_{n}})$$ by
$$ L_{B_{n}}(\beta)(i_{0},i_{1},...,
i_{q})=\del_{B_{n}}^{i_{0}}(\beta(i_{0},i_{1},..., i_{q})) \qquad
\forall \beta \in
\check{C}^{q}(\mathbf{U},\Omega^{p;\omega}_{X_{n}}),$$  where
$i_{0}<i_{1}...< i_{q}$ (see \cite{Ye7}).\vskip 0.1cm
 Define, for each $U_{i}$ the
total interior product with respect to $B_{n}$, $$I^{i}:
\Omega^{p;\omega}_{X_{n}}(U_{i}) \oplus \Omega^{p;\omega}_{\bar{X}_{n}}(U_{i})\rightarrow \Omega^{p;\omega}_{X_{n}}(U_{i}) \oplus \Omega^{p;\omega}_{\bar{X}_{n}}(U_{i})$$
by
$$ I^{i}(\mu_1 \del_{X_n}g_{1}\wedge \del_{X_n}g_{2}\wedge...\wedge \del_{X_n}g_{p}+\mu_2 \del_{\bar{X}_n}{g^{'}}_{1}\wedge \del_{\bar{X}_n}g^{'}_{2}\wedge...\wedge \del_{\bar{X}_n}g^{'}_{p})= $$
$$
\mu_1 \sum_{j=1}^{p}\del_{X_n}g_{1}\wedge...\wedge \del_{X_n}g_{j-1}\wedge \del^{i}_{B_{n}}(g_{j})\wedge \del_{X_n}g_{j+1}\wedge...\wedge \del_{X_n}g_{p}+ $$ $$\mu_2 \sum_{j=1}^{p}\del_{\bar{X}_n}g^{'}_{1}\wedge...\wedge \del_{\bar{X}_n}g^{'}_{j-1}\wedge \del^{i}_{B_{n}}(g^{'}_{j})\wedge \del_{\bar{X}_n}g^{'}_{j+1}\wedge...\wedge \del_{\bar{X}_n}g^{'}_{p}.$$
When $p=0$, we put $I^{i}=0$ (see \cite{Ye7}).\vskip 0.1cm Define
$\lambda: \check{C}^{q}(\mathbf{U},\Omega^{p;\omega}_{X_{n}} \oplus \Omega^{p;\omega}_{\bar{X}_{n}})
\rightarrow \check{C}^{q+1}(\mathbf{U},\Omega^{p;\omega}_{X_{n}}\oplus \Omega^{p;\omega}_{\bar{X}_{n}}) $ by
$$ (\lambda
\beta)(i_{0},...,i_{q+1})=(I^{i_{0}}-I^{i_{1}})\beta(i_{1},...,i_{q+1}) \qquad \forall \beta \in \check{C}^{q}(\mathbf{U},\Omega^{p;\omega}_{X_{n}}\oplus \Omega^{p;\omega}_{\bar{X}_{n}}).$$

An we have the following lemma, the prove is completely the same as lemma 3.1 in \cite{ye}:
\begin{lemma}
$$\lambda \circ \varphi \equiv \delta \circ \varphi - \varphi \circ
\delta.$$ 
\end{lemma}
With the above preparation, we are ready to  study the jumping phenomenon of
the dimensions of Bott-Chern or Aeppli cohomology groups, for arbitrary $[\theta]$ belongs to
$\HH^l(X,\sB_{p,q})$, suppose we can extend $[\theta]$ to
order $n-1$ in $\HH^l(X_{n-1},\sB_{p,q;X_{n-1}/B_{n-1}}^{\bullet}) $. Denote
such element by $[\theta_{n-1}]$. In the following, we try to find
out the obstruction of the extension of $[\theta_{n-1}]$ to $n$th
order. Consider the exact sequence\\
$$ 0 \rightarrow \mathcal{M}^{n}_{0} / \mathcal{M}^{n+1}_{0}\otimes \sB_{p,q;X_{0}/B_{0}}^{\bullet}
\rightarrow \sB_{p,q;X_{n}/B_{n}}^{\bullet} \rightarrow
\sB_{p,q;X_{n-1}/B_{n-1}}^{\bullet}\rightarrow 0$$ which induces a long
exact
sequence\\
$$ 0 \rightarrow \HH^0(X_n,\mathcal{M}^{n}_{0} / \mathcal{M}^{n+1}_{0}\otimes
\sB_{p,q;X_{0}/B_{0}}^{\bullet})\rightarrow \HH^0(X_{n},
\sB_{p,q;X_{n}/B_{n}}^{\bullet})\rightarrow
\HH^0(X_{n-1},\sB_{p,q;X_{n-1}/B_{n-1}}^{\bullet})
$$
$$ \rightarrow \HH^1(X_n,\mathcal{M}^{n}_{0} / \mathcal{M}^{n+1}_{0}\otimes
\sB_{p,q;X_{0}/B_{0}}^{\bullet}) \rightarrow ... .$$ Let $[\theta]$ be a cohomology class in $\HH^l(X,\sB_{p,q;X_{0}/B_{0}}^{\bullet})$.

 The obstruction for
$[\theta_{n-1}]$ comes from the non trivial image of the
connecting homomorphism
$\delta^{*}: \HH^l(X_{n-1},\sB_{p,q;X_{n-1}/B_{n-1}}^{\bullet}) \rightarrow \HH^{l+1}(X_n,\mathcal{M}^{n}_{0} / \mathcal{M}^{n+1}_{0}\otimes
\sB_{p,q;X_{0}/B_{0}}^{\bullet})
$. \\
Now we are ready to calculate the formula for the obstructions.
Let $\tilde{\theta}$ be an element of
$\check{C}^{l}(\mathbf{U},\sB_{p,q;X_{n}/B_{n}}^{\bullet})$ such that
its quotient image in
$\check{C}^{l}(\mathbf{U},\sB_{p,q;X_{n-1}/B_{n-1}}^{\bullet})$ is
$\theta_{n-1}$. Then $\delta^{*}([\theta_{n-1}])$=
$[\dcechg (\tilde{\theta})]$ which is an element of
$$\HH^{l+1}(X_n,\mathcal{M}_{0}^{n} / \mathcal{M}^{n+1}_{0}\otimes
\sB_{p,q;X_{0}/B_{0}}^{\bullet}) \cong \mathrm{m}_{0}^{n} /
\mathrm{m}^{n+1}_{0} \otimes
\HH^{l+1}(X,\sB_{p,q;X_{0}/B_{0}}^{\bullet}).$$  Denote
$r_{X_{n}}$ the restriction to the space $X_{n}^{\omega}$(topological space $X$ with structure sheaf $\sO_{X_{n}}^{\omega}$)
and
denote the following complex
 $$ \pi^{-1}(\Omega_{B_{n}|B_{n-1}}^{\omega}) \stackrel{(+,-)}{\to} \pi_{n-1}^{*}(\Omega_{B_{n}|B_{n-1}})^{\omega} \otimes \sO^{\omega}_{X_n} \oplus \bar{\pi}_{n-1}^{*}(\bar{\Omega}_{B_{n}|B_{n-1}})^{\omega} \otimes \sO^{\omega}_{\bar{X}_n} $$$$\to  \pi_{n-1}^{*}(\Omega_{B_{n}|B_{n-1}})^{\omega} \wedge   \Omega_{X_n/B_n}^{1;\omega}\oplus \bar{\pi}_{n-1}^{*}(\bar{\Omega}_{B_{n}|B_{n-1}})^{\omega} \wedge \bar{\Omega}_{X_n/B_n}^{1;\omega}$$$$\to \ldots \pi_{n-1}^{*}(\Omega_{B_{n}|B_{n-1}})^{\omega} \wedge  \Omega_{X_n/B_n}^{p-1;\omega}\oplus  \bar{\pi}_{n-1}^{*}(\bar{\Omega}_{B_{n}|B_{n-1}})^{\omega} \wedge \bar{\Omega}_{X_n/B_n}^{p-1;\omega}$$
$$\to  \bar{\pi}_{n-1}^{*}(\bar{\Omega}_{B_{n}|B_{n-1}})^{\omega} \wedge  \bar{\Omega}_{X_n/B_n}^{p;\omega}\to\ldots\to \bar{\pi}_{n-1}^{*}(\bar{\Omega}_{B_{n}|B_{n-1}})^{\omega} \wedge\bar{\Omega}_{X_n/B_n}^{q-1;\omega}\to 0,$$
by
$\pi_{n-1}^{*}(\Omega_{B_{n}|B_{n-1}})\wedge\sB_{p,q;X_{n-1}/B_{n-1}}^{\bullet.}$

In order
to give the obstructions an explicit calculation, we need to
consider the following map $$ \rho:
\HH^{l}(X_n,\mathcal{M}^{n}_{0}/\mathcal{M}^{n+1}_{0}\otimes
\sB_{p,q;X_{0}/B_{0}}^{\bullet}) \rightarrow
\HH^{l}(X_{n-1},\pi_{n-1}^{*}(\Omega_{B_{n}|B_{n-1}})\wedge\sB_{p,q;X_{n-1}/B_{n-1}}^{\bullet})
$$ which is defined by $\rho[\sigma]=[ \varphi^{-1} \circ
r_{X_{n-1}}\circ (L_{B_{n}}+L_{\bar{B}_{n}})\circ \varphi (\sigma)]$, where $\varphi^{-1}$ is the quotient maps:  $\check{C}^{\bullet}(\mathbf{U}, \pi_{n-1}^{*}(\Omega_{B_{n}|B_{n-1}})^{\omega} \wedge(\Omega^{p;\omega}_{X_{n}|X_{n-1}} \oplus {\Omega}^{p;\omega}_{\bar {X}_{n}|\bar{X}_{n-1}} ))
\rightarrow \check{C}^{\bullet}(\mathbf{U}, \pi_{n-1}^{*}(\Omega_{B_{n}|B_{n-1}}) ^{\omega}\wedge(\Omega^{p;\omega}_{X_{n-1}/B_{n-1}} \oplus \bar{\Omega}^{p;\omega}_{X_{n-1}/B_{n-1}} ) ).$
An we have the following lemmas, the proof is completely the same as lemma 3.2 and lemma 3.3 in \cite{ye}:
\begin{lemma}
The map $\rho
$ is well defined.
\end{lemma}
\begin{lemma}
$\rho([\dcechg(\tilde{\theta})])$ is exactly
$o_{n}([\theta])$ in the previous section.
\end{lemma}
Now consider the
following exact sequence. The connecting homomorphism  of the
associated long exact sequence gives the Kodaira-Spencer
 class of order $n$ \cite{Ye4} [1.3.2],\\
\begin{align}
  0\rightarrow \pi_{n-1}^{*}(\Omega_{B_{n}|B_{n-1}})^{\omega}\rightarrow
 \Omega_{X_{n}|X_{n-1}}^{\omega}\rightarrow
 \Omega_{X_{n-1}/B_{n-1}}^{\omega}\rightarrow 0. \label{exactsq1}
\end{align}
By wedge the above exact sequence with
$\Omega^{p-1;\omega}_{X_{n-1}/B_{n-1}}$, we get a new exact sequence. The
connecting homomorphism of such exact sequence gives us a map from
$H^{q}(X_{n-1},\Omega^{p;\omega}_{X_{n-1}/B_{n-1}})$ to
$H^{q+1}(X_{n-1},\pi^{*}(\Omega_{B_{n}|B_{n-1}})^{\omega}\wedge\Omega^{p-1;\omega}_{X_{n-1}/B_{n-1}})$.
Denote such map by $\kappa_{n}\llcorner$, for such map is simply
the inner product with the Kodaira-Spencer
 class of order $n$.
With the above preparation, we are ready to proof the main
theorem of this section.
\begin{theorem} \label{main}
Let $\pi:\mathcal{X}\rightarrow B$ be a deformation of
$\pi^{-1}(0)=X$, where $X$ is a compact complex manifold. Let
$\pi_{n}:X_{n}\rightarrow B_{n}$ be the $n$th order deformation of
$X$. For arbitrary $[\theta]$ belongs to $\HH^l(X,\sB_{p,q}^{\bullet})$,
suppose we can extend $[\theta]$ to order $n-1$ in
$H^l(X_{n-1},\sB_{p,q;X_{n-1}/B_{n-1}}^{\bullet})$. Denote such element by
$[\theta_{n-1}]$. The obstruction of the extension of $[\theta]$
to $n$th order is given by:
$$ o_{n}([\theta])=-\partial^{\bar{\partial},\sB}_{X_{n-1}/B_{n-1}} \circ \kappa_{n} \llcorner \circ \partial^{\sB,\bar{\partial}}_{X_{n-1}/B_{n-1}}([\theta_{n-1}])- \bar{\partial}^{{\partial},\sB}_{X_{n-1}/B_{n-1}} \circ \bar{\kappa}_{n} \llcorner \circ \bar{\partial}^{\sB,{\partial}}_{X_{n-1}/B_{n-1}}([\theta_{n-1}]),$$
where $\kappa_{n}$ is the $n$th order Kodaira-Spencer class and $\bar{\kappa}_{n}$  is
the $n$th order Kodaira-Spencer class of the deformation $ \bar{\pi}: \bar{\sX} \rightarrow \bar{B}$.
$\partial^{\bar{\partial},\sB}_{X_{n-1}/B_{n-1}}$, $\bar{\partial}^{{\partial},\sB}_{X_{n-1}/B_{n-1}}$ , $\partial^{\sB,\bar{\partial}}_{X_{n-1}/B_{n-1}}$ and $\bar{\partial}^{\sB,{\partial}}_{X_{n-1}/B_{n-1}}$
are the  maps defined in $\S2$.
\end{theorem}

 \begin{proof}
Note that $o_{n}([\theta]) = \rho \circ \delta(\tilde{\theta}) = [\varphi^{-1} \circ r_{X_{n-1}}\circ (L_{B_{n}}+L_{\bar{B}_{n}} )\circ \varphi \circ \delta(\tilde{\theta})]  .$ Because $ (L_{B_{n}}+L_{\bar{B}_{n}}+\del_{X_{n}/B_{n}}+\delbar_{X_{n}/B_{n}} )\circ \dcechg = - \dcechg \circ (L_{B_{n}}+L_{\bar{B}_{n}}+\del_{X_{n}/B_{n}}+\delbar_{X_{n}/B_{n}} ).$
\begin{eqnarray*}
r_{X_{n-1}}\circ (L_{B_{n}}+L_{\bar{B}_{n}} ) \circ \varphi \circ
\dcechg(\tilde{\theta}) & \equiv & r_{X_{n-1}} \circ (L_{B_{n}}+L_{\bar{B}_{n}} )
\circ (\dcechg \circ \varphi - \lambda \circ
\varphi)(\tilde{\theta})  \\
& \equiv & r_{X_{n-1}} \circ (L_{B_{n}}+L_{\bar{B}_{n}} ) \circ \dcechg  \circ \varphi
(\tilde{\theta})  \\
& \equiv & -r_{X_{n-1}} \circ (\del^{\bullet}_{X_{n}/B_{n}} \circ
\dcechg + \dcechg  \circ \del^{\bullet}_{X_{n}/B_{n}}+ \\
&& \delbar^{\bullet}_{X_{n}/B_{n}} \circ
\dcechg + \dcechg  \circ \delbar^{\bullet}_{X_{n}/B_{n}} +\dcechg  \circ
(L_{B_{n}}+L_{\bar{B}_{n}} ) ) \circ
\varphi (\tilde{\theta}) \\
& \equiv & -r_{X_{n-1}} \circ (\del^{\bullet}_{X_{n}/B_{n}} \circ
\dcechg + \dcechg \circ \del^{\bullet}_{X_{n}/B_{n}} + \delbar^{\bullet}_{X_{n}/B_{n}} \\&&\circ
\dcechg + \dcechg \circ \delbar^{\bullet}_{X_{n}/B_{n}}) \circ \varphi
(\tilde{\theta})
-  \dcechg \circ r_{X_{n-1}}\circ (L_{B_{n}}+L_{\bar{B}_{n}} )\circ
\varphi (\tilde{\theta}). \\
\end{eqnarray*}

\noindent Therefore
\begin{eqnarray*}
[\varphi^{-1} \circ r_{X_{n-1}}\circ (L_{B_{n}}+L_{\bar{B}_{n}} )\circ \varphi \circ
\delta(\tilde{\theta})] & = & [-\varphi^{-1} \circ r_{X_{n-1}} \circ (\del^{\bullet}_{X_{n}/B_{n}} \circ
\dcechg + \dcechg \circ \del^{\bullet}_{X_{n}/B_{n}}\\&& + \delbar^{\bullet}_{X_{n}/B_{n}} \circ
\dcechg + \dcechg \circ \delbar^{\bullet}_{X_{n}/B_{n}}) \circ \varphi
(\tilde{\theta}) ]\\
& = & -[ \del_{X_{n-1}/B_{n-1}} \circ \varphi^{-1} \circ r_{X_{n-1}} \circ \dcechg \circ \varphi
(\tilde{\theta}) +\\&& \delbar_{X_{n-1}/B_{n-1}} \circ\varphi^{-1} \circ  r_{X_{n-1}} \circ \dcechg \circ \varphi
(\tilde{\theta}) \\
&& + \varphi^{-1} \circ r_{X_{n-1}} \circ \dcechg \circ \varphi (\widetilde{\del_{X_{n-1}/B_{n-1}}(\theta_{n-1})}) + \\&& \varphi^{-1} \circ r_{X_{n-1}} \circ \dcechg \circ \varphi (\widetilde{\delbar_{X_{n-1}/B_{n-1}}(\theta_{n-1})})\\
\end{eqnarray*}
Since $(\varphi^{-1} \circ  r_{X_{n-1}} \circ \dcechg \circ \varphi(\tilde{\theta}))^{p-1,0}_u=0$ and $(\varphi^{-1} \circ r_{X_{n-1}} \circ \dcechg \circ \varphi (\tilde{\theta}))^{0,q-1}_v=0$, by lemma \ref{lemma4} and lemma \ref{lemma5}, we know that $[\del_{X_{n-1}/B_{n-1}} \circ \varphi^{-1} \circ r_{X_{n-1}} \circ \dcechg \circ \varphi(\tilde{\theta})]=0$ and $[\delbar_{X_{n-1}/B_{n-1}} \circ\varphi^{-1} \circ  r_{X_{n-1}} \circ \dcechg \circ \varphi(\tilde{\theta})]=0$.
And from lemma \ref{lemma4} and lemma \ref{lemma5}, we also know that
\begin{eqnarray*}
[\varphi^{-1} \circ r_{X_{n-1}} \circ \dcechg \circ \varphi (\widetilde{\del_{X_{n-1}/B_{n-1}}(\theta_{n-1})})]
&=& [\varphi^{-1} \circ r_{X_{n-1}} \circ \dcechg \circ \varphi (\widetilde{\del_{X_{n-1}/B_{n-1}}(\theta_{n-1} -\theta_{n-1;u}^{p-1,0})} \\
&&+\del_{X_{n-1}/B_{n-1}}\theta_{n-1;u}^{p-1,0})]\\
&=& [\varphi^{-1} \circ r_{X_{n-1}} \circ \dcechg \circ \varphi (\widetilde{\del_{X_{n-1}/B_{n-1}}\theta_{n-1;u}^{p-1,0}})]
\end{eqnarray*}
and
\begin{eqnarray*}
[\varphi^{-1} \circ r_{X_{n-1}} \circ \dcechg \circ \varphi (\widetilde{\delbar_{X_{n-1}/B_{n-1}}(\theta_{n-1})})]
&=& [\varphi^{-1} \circ r_{X_{n-1}} \circ \dcechg \circ \varphi (\widetilde{\delbar_{X_{n-1}/B_{n-1}}(\theta_{n-1}-\theta_{n-1;v}^{0,q-1})}\\ && + \delbar_{X_{n-1}/B_{n-1}}\theta_{n-1;v}^{0,q-1})]\\
&=& [\varphi^{-1} \circ r_{X_{n-1}} \circ \dcechg \circ \varphi (\widetilde{\delbar_{X_{n-1}/B_{n-1}}\theta_{n-1;v}^{0,q-1}})]
\end{eqnarray*}
By the definition of the maps: $\partial^{\bar{\partial},\sB}_{X_{n-1}/B_{n-1}}$ , $\partial^{\sB,\bar{\partial}}_{X_{n-1}/B_{n-1}}$ and Lemma 3.4 in \cite{ye}, we have
\begin{eqnarray*}
[\varphi^{-1} \circ r_{X_{n-1}} \circ \dcechg \circ \varphi (\widetilde{\del_{X_{n-1}/B_{n-1}}\theta_{n-1;u}^{p-1,0}})]
={\partial}^{\bar{\partial},\sB}_{X_{n-1}/B_{n-1}} \circ{\kappa}_{n} \llcorner \circ {\partial}^{\sB,\bar{\partial}}_{X_{n-1}/B_{n-1}}([\theta_{n-1}])
\end{eqnarray*}
and similarly, we have
\begin{eqnarray*}
[\varphi^{-1} \circ r_{X_{n-1}} \circ \dcechg \circ \varphi (\widetilde{\delbar_{X_{n-1}/B_{n-1}}\theta_{n-1;v}^{0,q-1}})]
= \bar{\partial}^{{\partial},\sB}_{X_{n-1}/B_{n-1}} \circ \bar{\kappa}_{n} \llcorner \circ \bar{\partial}^{\sB,{\partial}}_{X_{n-1}/B_{n-1}}([\theta_{n-1}])
\end{eqnarray*}
So we have:
\begin{eqnarray*}
[\varphi^{-1} \circ r_{X_{n-1}}\circ (L_{B_{n}}+L_{\bar{B}_{n}} ) \circ \varphi \circ
\delta(\tilde{\theta})]
 = -{\partial}^{\bar{\partial},\sB}_{X_{n-1}/B_{n-1}} \circ{\kappa}_{n}\llcorner \circ {\partial}^{\sB,\bar{\partial}}_{X_{n-1}/B_{n-1}}([\theta_{n-1}])\\-\bar{\partial}^{{\partial},\sB}_{X_{n-1}/B_{n-1}} \circ \bar{\kappa}_{n} \llcorner \circ \bar{\partial}^{\sB,{\partial}}_{X_{n-1}/B_{n-1}}([\theta_{n-1}]).
\end{eqnarray*}
\end{proof}
Apply the above theorem and theorem \ref{main0} to study the jumping phenomenon of the dimensions of
Bott-Chern(Aeppli) cohomology groups. We have the following theorems.

\begin{theorem} \label{main4}
Let $\pi:\mathcal{X}\rightarrow B$ be a deformation of
$\pi^{-1}(0)=X$, where $X$ is a compact complex manifold. Let
$\pi_{n}:X_{n}\rightarrow B_{n}$ be the $n$th order deformation of
$X$. If there exists an elements $[\theta^1]$ in $H^{p,q}_{\rm BC}(X)$ or
an elements $[\theta^2]$ in $H^{p-1,q-1}_{\rm A}(X)$ and
a minimal natural number $n \geq 1$ such that the $n$th order obstruction
$o_n([\theta^1])\neq 0$ or   $o_n([\theta^2])\neq 0$, then the $h_{p,q}^{\mathrm{\rm BC}}(X(t))$
will jump at the point $t=0$. And the formulas for the obstructions are given by:
$$ o_n([\theta^1]) = -{\partial}^{\bar{\partial},\sB}_{X_{n-1}/B_{n-1}} \circ{\kappa}_{n}\llcorner \circ r_{BC,\bar{\partial}}([\theta_{n-1}])\\-\bar{\partial}^{{\partial},\sB}_{X_{n-1}/B_{n-1}} \circ \bar{\kappa}_{n} \llcorner \circ r_{BC,{\partial}}([\theta_{n-1}]);$$
$$ o_n([\theta^2]) = - \partial^{\bar{\partial},BC}_{X_{n-1}/B_{n-1}} \circ{\kappa}_{n}\llcorner \circ \partial^{A,\bar{\partial}}_{X_{n-1}/B_{n-1}}([\theta_{n-1}])\\-\bar{\partial}^{{\partial},BC}_{X_{n-1}/B_{n-1}} \circ \bar{\kappa}_{n} \llcorner \circ \bar{\partial}^{A,{\partial}}_{X_{n-1}/B_{n-1}}([\theta_{n-1}]) .$$
 \end{theorem}
 \begin{theorem} \label{main5}
Let $\pi:\mathcal{X}\rightarrow B$ be a deformation of
$\pi^{-1}(0)=X$, where $X$ is a compact complex manifold. Let
$\pi_{n}:X_{n}\rightarrow B_{n}$ be the $n$th order deformation of
$X$. If there exists an elements $[\theta^1]$ in $H^{p,q}_{\rm A}(X)$ or
an elements $[\theta^2]$ in $\HH^{p+q}(X,\sB_{p+1,q+1}^{\bullet})$ and
a minimal natural number $n \geq 1$ such that the $n$th order obstruction
$o_n([\theta^1])\neq 0$ or   $o_n([\theta^2])\neq 0$, then the $h_{p,q}^{\mathrm{\rm BC}}(X(t))$
will jump at the point $t=0$. And the formulas for the obstructions are given by:
$$ o_n([\theta^1]) = - \partial^{\bar{\partial},BC}_{X_{n-1}/B_{n-1}} \circ{\kappa}_{n}\llcorner \circ \partial^{A,\bar{\partial}}_{X_{n-1}/B_{n-1}}([\theta_{n-1}])\\-\bar{\partial}^{{\partial},BC}_{X_{n-1}/B_{n-1}} \circ \bar{\kappa}_{n} \llcorner \circ \bar{\partial}^{A,{\partial}}_{X_{n-1}/B_{n-1}}([\theta_{n-1}]) .$$
$$ o_n([\theta^2]) = - r_{\bar{\partial},A}  \circ{\kappa}_{n} \llcorner\circ {\partial}^{\sB,\bar{\partial}}_{X_{n-1}/B_{n-1}}([\theta_{n-1}])\\-r_{{\partial},A}\circ \bar{\kappa}_{n}\llcorner \circ \bar{\partial}^{\sB,{\partial}}_{X_{n-1}/B_{n-1}}([\theta_{n-1}]) .$$
 \end{theorem}
 By these theorems, we can get the following corollaries immediately.
 \begin{corollary} \label{ap1}
Let $\pi:\mathcal{X} \rightarrow B$ be a deformation of
$\pi^{-1}(0)=X$, where $X$ is a compact complex manifold. Suppose
that up to order $n$, the maps $r_{BC,\bar{\partial}}: H^{p,q}_{\rm BC}(X_n/B_n) \rightarrow
H^q(X_n,\Omega_{X_n/B_n}^{p;\omega}) $ and $r_{BC,{\partial}}: H^{p,q}_{\rm BC}(X_n/B_n) \rightarrow
H^p(\bar{X}_n,\bar{\Omega}_{X_n/B_n}^{q;\omega}) $ is 0. For arbitrary $[\theta]$ that belongs to
$ H^{p,q}_{\rm BC}(X) $, it can be extended to order $n+1$ in  $H^{p,q}_{\rm BC}(X_{n+1}/B_{n+1}).$
\end{corollary}
\begin{proof}
This result can be shown by induction on $k$. \\
\indent Suppose that the corollary is proved for $k-1$, then we
can extend $[\theta]$ to and element $[\theta_{k-1}]$ in
$H^{p,q}_{\rm BC}(X_{k-1}/B_{k-1}).$ By Theorem \ref{main4} ,
the obstruction for the extension of $[\theta]$ to $k$th order
comes from:
$$ o_k([\theta]) = -{\partial}^{\bar{\partial},\sB}_{X_{k-1}/B_{k-1}} \circ{\kappa}_{n}\llcorner \circ r_{BC,\bar{\partial}}([\theta_{k-1}])\\-\bar{\partial}^{{\partial},\sB}_{X_{k-1}/B_{k-1}} \circ \bar{\kappa}_{n} \llcorner \circ r_{BC,{\partial}}([\theta_{k-1}]);$$
By the assumption,
$r_{BC,\bar{\partial}}: H^{p,q}_{\rm BC}(X_{k-1}/B_{k-1}) \rightarrow
H^q(X_{k-1},\Omega_{X_{k-1}/B_{k-1}}^{p;\omega}) $ and $r_{BC,{\partial}}: H^{p,q}_{\rm BC}(X_{k-1}/B_{k-1}) \rightarrow
H^p(\bar{X}_{k-1},\bar{\Omega}_{X_{k-1}/B_{k-1}}^{q;\omega}) $ is 0, where $k\leq n+1$
. So we have
$r_{BC,{\partial}}([\theta_{k-1}])=0$ and  $r_{BC,{\bar{\partial}}}([\theta_{k-1}])=0$. So the obstruction
$o_{k}([\theta])$ is trivial which means $[\theta]$ can be
extended to $k$th order.
\end{proof}

Since we have $\partial^{A,\bar{\partial}}_{X_n/B_n}: H^{p,q}_{\rm A}(X_n/B_n) \rightarrow H^q(X_n,\Omega_{X_n/B_n}^{p+1;\omega}) $ is the composition of $\partial^{A,BC}_{X_n/B_n}: H^{p,q}_{\rm A}(X_n/B_n)
\rightarrow   H^{p+1,q}_{\rm BC}(X_n/B_n)$ and  $r_{BC,\bar{\partial}}: H^{p+1,q}_{\rm BC}(X_n/B_n) \rightarrow
H^q(X_n,\Omega_{X_n/B_n}^{p+1;\omega}). $ With the same proof of the above corollary, we have the following result and we omit the proof.
\begin{corollary} \label{ap1}
Let $\pi:\mathcal{X} \rightarrow B$ be a deformation of
$\pi^{-1}(0)=X$, where $X$ is a compact complex manifold. Suppose
that up to order $n$, the maps $r_{BC,\bar{\partial}}: H^{p+1,q}_{\rm BC}(X_n/B_n) \rightarrow
H^q(X_n,\Omega_{X_n/B_n}^{p+1;\omega}) $ and $r_{BC,{\partial}}: H^{p,q+1}_{\rm BC}(X_n/B_n) \rightarrow
H^p(\bar{X}_n,\bar{\Omega}_{X_n/B_n}^{q+1;\omega}) $ is 0. For arbitrary $[\theta]$ that belongs to
$ H^{p,q}_{\rm A}(X) $, it can be extended to order $n+1$ in  $H^{p,q}_{\rm A}(X_{n+1}/B_{n+1}).$
\end{corollary}

\section{An Example}\label{section6}
\renewcommand{\theequation}
{4.\arabic{equation}} \setcounter{equation}{-1}
In this section, we will use the formula in
previous section to study the jumping phenomenon of the dimensions of Bott-Chern cohomology groups $h_{\rm BC}^{p,q}$ and Aeppli cohomology groups $h_{\rm A}^{p,q}$ of small deformations of Iwasawa manifold. It
was Kodaira who first calculated small deformations of Iwasawa
manifold \cite{Ye2}. In the first part of this section, let us
recall his result. \vskip 0.1cm \indent Set
\begin{displaymath}
G=\left\{\left( \begin{array}{ccc} 1 & z_{2} & z_{3}\\
                            0 & 1 & z_{1}\\
                            0 & 0 & 1\\
                            \end{array} \right); z_{i}\in
                            \mathbb{C}\right\}\cong
                            \mathbb{C}^{3},
\end{displaymath}\\
\begin{displaymath}
\Gamma=\left\{\left( \begin{array}{ccc} 1 & \omega_{2} & \omega_{3}\\
                            0 & 1 & \omega_{1}\\
                            0 & 0 & 1\\
                            \end{array} \right); \omega_{i}\in
                            \mathbb{Z}+\mathbb{Z}\sqrt{-1}
                            \right\}\\.
\end{displaymath}\\
The multiplication is defined by
\begin{displaymath}
\left( \begin{array}{ccc} 1 & z_{2} & z_{3}\\
                            0 & 1 & z_{1}\\
                            0 & 0 & 1\\
                            \end{array} \right)
                            \left( \begin{array}{ccc} 1 & \omega_{2} & \omega_{3}\\
                            0 & 1 & \omega_{1}\\
                            0 & 0 & 1\\
                            \end{array} \right)=
                            \left( \begin{array}{ccc} 1 & z_{2}+\omega_{2} & z_{3}+\omega_{1}z_{2}+\omega_{3}\\
                            0 & 1 & z_{1}+\omega_{1}\\
                            0 & 0 & 1\\
                            \end{array} \right).
\end{displaymath}\\
$X=G/\Gamma$ is called Iwasawa manifold. We may consider
$X=\mathbb{C}^{3}/\Gamma$. $g\in \Gamma$ operates on
$\mathbb{C}^3$ as follows: \\
$$ z'_{1}=z_{1}+\omega_{1},   \qquad    z'_{2}=z_{2}+\omega_{2},  \qquad
z'_{3}=z_{3}+\omega_{1}z_{2}+\omega_{3},
$$\\
where $g=(\omega_{1},\omega_{2},\omega_{3})$ and $z'=z\cdot g$.
There exist holomorphic $1$-froms
$\varphi_{1},\varphi_{2},\varphi_{3}$ which are linearly
independent at every point on $X$ and are given by \\
$$\varphi_{1}=dz_{1},\qquad \varphi_{2}=dz_{2}, \qquad
\varphi_{3}=dz_{3}-z_{1}dz_{2},$$\\
so that
$$ d\varphi_{1}=d\varphi_{2}=0, \qquad
d\varphi_{3}=-\varphi_{1}\wedge\varphi_{2}.$$\\
On the other hand we have holomorphic vector fields
$\theta_{1},\theta_{2},\theta_{3}$ on $X$ given by
$$ \theta_{1}=\frac{\partial}{\partial z_{1}},\qquad \theta_{2}=
\frac{\partial}{\partial z_{2}}+z_{1}\frac{\partial}{\partial
z_{3}},\qquad \theta_{3}=\frac{\partial}{\partial z_{3}}.$$ \\
It is easily seen that \\
$$ [\theta_{1},\theta_{2}]=-[\theta_{2},\theta_{1}]=\theta_{3}, \qquad
[\theta_{1},\theta_{3}]=[\theta_{2},\theta_{3}]=0.$$ In view of
Theorem $3$ in \cite{Ye2}, $H^{1}(X,\mathcal{O}_{X})$ is spanned
by $\overline{\varphi}_{1},\overline{\varphi}_{2}$. Since $\Theta$
is isomorphic to $\mathcal{O}^{3}$, $H^{1}(X,T_{X})$ is spanned by
$\theta_{i}\overline{\varphi}_{\lambda}, i=1,2,3, \lambda=1,2$.\\
\indent Consider the small deformation of $X$ given by
$$ \psi(t)=\sum^{3}_{i=1}\sum_{\lambda=1}^{2}t_{i\lambda}\theta_{i}\overline{\varphi}_{\lambda}t-
(t_{11}t_{22}-t_{21}t_{12})\theta_{3}\overline{\varphi}_{3}t^{2}.$$ \\
We summarize the numerical characters of deformations. The
deformations are divided into the following three classes,
the classes and subclasses of this classification are characterized by the following values of the parameters(all the details can be found in \cite{angella-1}):\\
\begin{description}
 \item[class {\itshape (i)}] $t_{11}=t_{12}=t_{21}=t_{22}=0$; \\
 \item[class {\itshape (ii)}] $D\left(\tempo\right)=0$ and $\left(t_{11},\,t_{12},\,t_{21},\,t_{22}\right)\neq \left(0,\,0,\,0,\,0\right)$:
    \begin{description}
     \item[subclass {\itshape (ii.a)}] $D\left(\tempo\right)=0$ and rank $S=1$;
     \item[subclass {\itshape (ii.b)}] $D\left(\tempo\right)=0$ and rank $S=2$;\\
    \end{description}
 \item[class {\itshape (iii)}] $D\left(\tempo\right)\neq 0$:
    \begin{description}
     \item[subclass {\itshape (iii.a)}] $D\left(\tempo\right)\neq 0$ and rank $S=1$;
     \item[subclass {\itshape (iii.b)}] $D\left(\tempo\right)\neq 0$ and rank $S=2$;\\
    \end{description}
\end{description}
the matrix $S$ is defined by
$$ S \;:=\;
\left(
\begin{array}{cccc}
 \overline{\sigma_{1\bar1}} & \overline{\sigma_{2\bar2}} & \overline{\sigma_{1\bar2}} & \overline{\sigma_{2\bar1}} \\
 \sigma_{1\bar1} & \sigma_{2\bar2} & \sigma_{2\bar1} & \sigma_{1\bar2}
\end{array}
\right)
$$
where $\sigma_{1\bar1},\,\sigma_{1\bar2},\,\sigma_{2\bar1},\,\sigma_{2\bar2}\in\C$ and $\sigma_{12}\in\C$ are complex numbers depending only on $\tempo$ such that
$$ d\varphi^3_\tempo \;=:\; \sigma_{12}\,\varphi^1_{\tempo}\wedge\varphi^2_{\tempo}+\sigma_{1\bar1}\,\varphi^1
_{\tempo}\wedge\bar\varphi^1_{\tempo}+\sigma_{1\bar2}\,\varphi^1_{\tempo}\wedge\bar\varphi^2_{\tempo}+\sigma_{2\bar1}\,
\varphi^2_{\tempo}\wedge\bar\varphi^1_{\tempo}+\sigma_{2
\bar2}\,\varphi^2_{\tempo}\wedge\bar\varphi^2_{\tempo} .$$
The first order asymptotic behaviour of $\sigma_{12},\,\sigma_{1\bar1},\,\sigma_{1\bar2},\,\sigma_{2\bar1},\,\sigma_{2\bar2}$ for $\tempo$ near $0$ is the following:
\begin{equation}\label{eq:struttura-asintotica}
\left\{
\begin{array}{rcl}
\sigma_{12} &=& -1 +\opiccolo{\tempo} \\[5pt]
\sigma_{1\bar1} &=& t_{21} +\opiccolo{\tempo}  \\[5pt]
\sigma_{1\bar2} &=& t_{22} +\opiccolo{\tempo}  \\[5pt]
\sigma_{2\bar1} &=& -t_{11} +\opiccolo{\tempo} \\[5pt]
\sigma_{2\bar2} &=& -t_{12} +\opiccolo{\tempo} \\[5pt]
\end{array}
\right.
\qquad \text{ for } \qquad \tempo\in\text{ classes {\itshape (i)}, {\itshape (ii)} and {\itshape (iii)}}
\;.
\end{equation}
The following tables are given by D.Angella in \cite{angella-1}.

Dimensions of the cohomologies of the Iwasawa manifold and of its small deformations:

\begin{center}
\begin{tabular}{c|ccccc}
$\mathbf{H^\bullet_{dR}}$ & $\mathbf{b_1}$ & $\mathbf{b_2}$ & $\mathbf{b_3}$ & $\mathbf{b_4}$ & $\mathbf{b_5}$ \\[5pt]
\hline
$\mathbb{I}_3$ and {\itshape (i)}, {\itshape (ii)}, {\itshape (iii)} & $4$ & $8$ & $10$ & $8$ & $4$
\end{tabular}
\end{center}

\vspace{12pt}

\begin{center}
\begin{tabular*}{16cm}{c|cc|ccc|cccc|ccc|cc}
$\mathbf{H^{\bullet\bullet}_{\delbar}}$ & $\mathbf{h^{1,0}_{\delbar}}$ & $\mathbf{h^{0,1}_{\delbar}}$ & $\mathbf{h^{2,0}_{\delbar}}$ & $\mathbf{h^{1,1}_{\delbar}}$ & $\mathbf{h^{0,2}_{\delbar}}$ & $\mathbf{h^{3,0}_{\delbar}}$ & $\mathbf{h^{2,1}_{\delbar}}$ & $\mathbf{h^{1,2}_{\delbar}}$ & $\mathbf{h^{0,3}_{\delbar}}$ & $\mathbf{h^{3,1}_{\delbar}}$ & $\mathbf{h^{2,2}_{\delbar}}$ & $\mathbf{h^{1,3}_{\delbar}}$ & $\mathbf{h^{3,2}_{\delbar}}$ & $\mathbf{h^{2,3}_{\delbar}}$ \\[5pt]
\hline
$\I_3$ and \itshape{(i)} & $3$ & $2$ & $3$ & $6$ & $2$ & $1$ & $6$ & $6$ & $1$ & $2$ & $6$ & $3$ & $2$ & $3$ \\[5pt]
\itshape{(ii)} & $2$ & $2$ & $2$ & $5$ & $2$ & $1$ & $5$ & $5$ & $1$ & $2$ & $5$ & $2$ & $2$ & $2$ \\[5pt]
\itshape{(iii)} & $2$ & $2$ & $1$ & $5$ & $2$ & $1$ & $4$ & $4$ & $1$ & $2$ & $5$ & $1$ & $2$ & $2$
\end{tabular*}
\end{center}

\vspace{12pt}

\begin{center}
\begin{tabular*}{16cm}{c|cc|ccc|cccc|ccc|cc}
$\mathbf{H^{\bullet\bullet}_{\textrm{BC}}}$ & $\mathbf{h^{1,0}_{\textrm{BC}}}$ & $\mathbf{h^{0,1}_{\textrm{BC}}}$ & $\mathbf{h^{2,0}_{\textrm{BC}}}$ & $\mathbf{h^{1,1}_{\textrm{BC}}}$ & $\mathbf{h^{0,2}_{\textrm{BC}}}$ & $\mathbf{h^{3,0}_{\textrm{BC}}}$ & $\mathbf{h^{2,1}_{\textrm{BC}}}$ & $\mathbf{h^{1,2}_{\textrm{BC}}}$ & $\mathbf{h^{0,3}_{\textrm{BC}}}$ & $\mathbf{h^{3,1}_{\textrm{BC}}}$ & $\mathbf{h^{2,2}_{\textrm{BC}}}$ & $\mathbf{h^{1,3}_{\textrm{BC}}}$ & $\mathbf{h^{3,2}_{\textrm{BC}}}$ & $\mathbf{h^{2,3}_{\textrm{BC}}}$ \\[5pt]
\hline
$\I_3$ and \itshape{(i)} & $2$ & $2$ & $3$ & $4$ & $3$ & $1$ & $6$ & $6$ & $1$ & $2$ & $8$ & $2$ & $3$ & $3$ \\[5pt]
\itshape{(ii.a)} & $2$ & $2$ & $2$ & $4$ & $2$ & $1$ & $6$ & $6$ & $1$ & $2$ & $7$ & $2$ & $3$ & $3$ \\[5pt]
\itshape{(ii.b)} & $2$ & $2$ & $2$ & $4$ & $2$ & $1$ & $6$ & $6$ & $1$ & $2$ & $6$ & $2$ & $3$ & $3$ \\[5pt]
\itshape{(iii.a)} & $2$ & $2$ & $1$ & $4$ & $1$ & $1$ & $6$ & $6$ & $1$ & $2$ & $7$ & $2$ & $3$ & $3$ \\[5pt]
\itshape{(iii.b)} & $2$ & $2$ & $1$ & $4$ & $1$ & $1$ & $6$ & $6$ & $1$ & $2$ & $6$ & $2$ & $3$ & $3$
\end{tabular*}
\end{center}

\vspace{12pt}

\begin{center}
\begin{tabular*}{16cm}{c|cc|ccc|cccc|ccc|cc}
$\mathbf{H^{\bullet\bullet}_{\textrm{A}}}$ & $\mathbf{h^{1,0}_{\textrm{A}}}$ & $\mathbf{h^{0,1}_{\textrm{A}}}$ & $\mathbf{h^{2,0}_{\textrm{A}}}$ & $\mathbf{h^{1,1}_{\textrm{A}}}$ & $\mathbf{h^{0,2}_{\textrm{A}}}$ & $\mathbf{h^{3,0}_{\textrm{A}}}$ & $\mathbf{h^{2,1}_{\textrm{A}}}$ & $\mathbf{h^{1,2}_{\textrm{A}}}$ & $\mathbf{h^{0,3}_{\textrm{A}}}$ & $\mathbf{h^{3,1}_{\textrm{A}}}$ & $\mathbf{h^{2,2}_{\textrm{A}}}$ & $\mathbf{h^{1,3}_{\textrm{A}}}$ & $\mathbf{h^{3,2}_{\textrm{A}}}$ & $\mathbf{h^{2,3}_{\textrm{A}}}$ \\[5pt]
\hline
$\I_3$ and \itshape{(i)} & $3$ & $3$ & $2$ & $8$ & $2$ & $1$ & $6$ & $6$ & $1$ & $3$ & $4$ & $3$ & $2$ & $2$ \\[5pt]
\itshape{(ii.a)} & $3$ & $3$ & $2$ & $7$ & $2$ & $1$ & $6$ & $6$ & $1$ & $2$ & $4$ & $2$ & $2$ & $2$ \\[5pt]
\itshape{(ii.b)} & $3$ & $3$ & $2$ & $6$ & $2$ & $1$ & $6$ & $6$ & $1$ & $2$ & $4$ & $2$ & $2$ & $2$ \\[5pt]
\itshape{(iii.a)} & $3$ & $3$ & $2$ & $7$ & $2$ & $1$ & $6$ & $6$ & $1$ & $1$ & $4$ & $1$ & $2$ & $2$ \\[5pt]
\itshape{(iii.b)} & $3$ & $3$ & $2$ & $6$ & $2$ & $1$ & $6$ & $6$ & $1$ & $1$ & $4$ & $1$ & $2$ & $2$
\end{tabular*}
\end{center}



\indent From the tables above, we know that the jumping phenomenon happens in $h^{2,0}_{\rm BC}$, $h^{0,2}_{\rm BC}$ and $h^{2,2}_{\rm BC}$ of Bott-Chern cohomology and symmetrically happens in $h^{3,1}_{\rm A}$, $h^{1,3}_{\rm A}$ and $h^{1,1}_{\rm A}$ of Aeppli cohomology. Now let us explain the jumping phenomenon of the dimensions of Bott-Chern cohomology and
 Aeppli cohomology by using the obstruction formula. From $\S4$ in
\cite{angella-1}, it follows that the Bott-Chern cohomology groups
in bi-degree $(2,0),(0,2),(2,2)$ are:
\begin{eqnarray*}
H^{2,0}_{\rm BC} (X)& =
&Span_{\mathbb{C}}\{[\varphi_{1}\wedge
\varphi_{2}],[\varphi_{2}\wedge \varphi_{3}],[\varphi_{3}\wedge
\varphi_{1}]\},\\ 
H^{0,2}_{\rm BC} (X)& =
&Span_{\mathbb{C}}\{[\overline\varphi_{1}\wedge
\overline\varphi_{2}],[\overline{\varphi}_{2}\wedge
\overline{\varphi}_{3}],[\overline{\varphi}_{3}\wedge
\overline{\varphi}_{1}]\}, \\
H^{2,2}_{\rm BC} (X)& =
&Span_{\mathbb{C}}\{[{\varphi}_{2}\wedge
{\varphi}_{3}\wedge\overline{\varphi}_{1}\wedge
\overline{\varphi}_{2}], [{\varphi}_{3}\wedge
{\varphi}_{1}\wedge\overline{\varphi}_{1}\wedge
\overline{\varphi}_{2}],\\&&
[\varphi_{1}\wedge
\varphi_{2}\wedge\overline{\varphi}_{2}\wedge
\overline{\varphi}_{3}],[\varphi_{2}\wedge
\varphi_{3}\wedge\overline{\varphi}_{2}\wedge
\overline{\varphi}_{3}],
[\varphi_{3}\wedge
\varphi_{1}\wedge\overline{\varphi}_{2}\wedge
\overline{\varphi}_{3}],\\&&[\varphi_{1}\wedge
\varphi_{2}\wedge\overline{\varphi}_{3}\wedge
\overline{\varphi}_{1}],[\varphi_{2}\wedge
\varphi_{3}\wedge\overline{\varphi}_{3}\wedge
\overline{\varphi}_{1}],[\varphi_{3}\wedge
\varphi_{1}\wedge\overline{\varphi}_{3}\wedge
\overline{\varphi}_{1}]\}, \\
\end{eqnarray*}
and the Aeppli cohomology groups
in bi-degree $(3,1),(1,3),(1,1)$ are:
\begin{eqnarray*}
H^{3,1}_{\rm A} (X)& =
&Span_{\mathbb{C}}\{[\varphi_{1}\wedge
\varphi_{2}\wedge\varphi_{3}\wedge\overline{\varphi}_{1}],[\varphi_{1}\wedge \varphi_{2}\wedge\varphi_{3}\wedge\overline{\varphi}_{2}],[\varphi_{1}\wedge
\varphi_{2}\wedge\varphi_{3}\wedge\overline{\varphi}_{3}]\},\\
H^{1,3}_{\rm A} (X)& =
&Span_{\mathbb{C}}\{[\varphi_{1}\wedge\overline\varphi_{1}\wedge
\overline\varphi_{2}\wedge\overline{\varphi}_{3}],[\varphi_{2}\wedge\overline{\varphi}_{1}\wedge
\overline{\varphi}_{2}\wedge\overline{\varphi}_{3}],[\varphi_{3}\wedge\overline{\varphi}_{1}\wedge
\overline{\varphi}_{2}\wedge\overline{\varphi}_{3}]\}, \\
H^{1,1}_{\rm A} (X)& =
&Span_{\mathbb{C}}\{[\varphi_{1}\wedge
\overline\varphi_{1}],[\varphi_{1}\wedge
\overline\varphi_{2}],[\varphi_{1}\wedge
\overline\varphi_{3}],[\varphi_{2}\wedge
\overline\varphi_{1}],\\&&[\varphi_{2}\wedge
\overline\varphi_{2}],[\varphi_{2}\wedge
\overline\varphi_{3}],[\varphi_{3}\wedge
\overline\varphi_{1}],[\varphi_{3}\wedge
\overline\varphi_{2}]\}
.\end{eqnarray*}
 For example, let us first consider $h^{2,0}_{\rm BC}$, in the ii) class of
 deformation. The Kodaira-Spencer class of the this deformation is $\psi_{1}(t)=\sum^{3}_{i=1}\sum_{\lambda=1}^{2}t_{i\lambda}\theta_{i}\overline{\varphi}_{\lambda}$, and
  $\bar{\psi}_{1}(t)=\sum^{3}_{i=1}\sum_{\lambda=1}^{2}\bar{t}_{i\lambda}\bar{\theta}_{i}{\varphi}_{\lambda}$,
 with $t_{11}t_{22}-t_{21}t_{12}=0$. It is easy to check that
 $o_{1}(\varphi_{1}\wedge \varphi_{2})=-\partial(int(\psi_{1}(t))(\varphi_{1}\wedge \varphi_{2}))-\bar{\partial}(int(\bar{\psi}_{1}(t))(\varphi_{1}\wedge \varphi_{2}))=0$, $o_{1}(t_{11}\varphi_{2}\wedge
 \varphi_{3}-t_{21}\varphi_{1}\wedge \varphi_{3})=-\partial((t_{11}t_{22}-t_{21}t_{12})\varphi_{3}\wedge
 \overline{\varphi}_{2})=0$, and $o_{1}(\varphi_{2}\wedge
 \varphi_{3})=t_{21}\varphi_{1}\wedge\varphi_{2}\wedge\overline{\varphi}_{1}+
 t_{22}\varphi_{1}\wedge\varphi_{2}\wedge\overline{\varphi}_{2}$, $o_{1}(\varphi_{1}\wedge
 \varphi_{3})=t_{11}\varphi_{1}\wedge\varphi_{2}\wedge\overline{\varphi}_{1}+
 t_{12}\varphi_{1}\wedge\varphi_{2}\wedge\overline{\varphi}_{2}$.
 Therefore, we have shown that for an element of the subspace
 $Span_{\mathbb{C}}\{[\varphi_{1}\wedge \varphi_{2}],[t_{11}\varphi_{2}\wedge
 \varphi_{3}-t_{21}\varphi_{1}\wedge \varphi_{3}]\}$, the first
 order obstruction is trivial, while, since
 $(t_{11},t_{12},t_{21},t_{22})\neq (0,0,0,0)$, at least one of
 the obstruction $o_{1}(\varphi_{2}\wedge
 \varphi_{3})$, $o_{1}(\varphi_{1}\wedge
 \varphi_{3})$ is non trivial which partly explain why the Hodge
 number $h^{2,0}_{\rm BC}$ jumps from 3 to 2. For another example, let us
 consider $h^{1,1}_{\rm A}$, in the ii) class of deformation. It is easy
 to check that all the first order obstructions of the cohomology classes are trivial.
 However, if we want to study the jumping phenomenon, we also need to consider
 the obstructions come from $\HH^{2}(X,\sB_{2,2}^{\bullet})$. It is easy to check that:
 $$ \HH^{2}(X,\sB_{2,2}^{\bullet})=Span_{\mathbb{C}}\{[\varphi_{3}],[\bar{\varphi}_{3}]\}.$$
 and
 $$ o_1({\varphi_{3}})=-t_{11}\varphi_{2}\wedge
 \bar{\varphi}_{1}-t_{12}\varphi_{2}\wedge
 \bar{\varphi}_{2}+t_{21}\varphi_{1}\wedge
 \bar{\varphi}_{1}+t_{22}\varphi_{1}\wedge
 \bar{\varphi}_{2};$$
  $$ o_1({\bar{\varphi}_{3}})=-\bar{t}_{11}\bar{\varphi}_{2}\wedge
 {\varphi}_{1}-\bar{t}_{12}\bar{\varphi}_{2}\wedge
 {\varphi}_{2}+\bar{t}_{21}\bar{\varphi}_{1}\wedge
 {\varphi}_{1}+\bar{t}_{22}\bar{\varphi}_{1}\wedge
 {\varphi}_{2}.$$
 Note that the first order of $S$ is
 $$
\left(
\begin{array}{cccc}
 -\overline{t}_{21} & -\overline{t}_{12} & \overline{t}_{22}& \overline{t}_{11} \\
 -t_{21} & -t_{21} & t_{11} & t_{22}
\end{array}
\right)
$$
If rank of the first order of $S=1$, then there exists $c_1,c_2$ such that
$$ o_1(c_1{\varphi}_{3} +c_2 \bar{\varphi}_{3})  \neq 0.$$
If rank of the first order of $S=2$, then for all $c_1,c_2$
$$ o_1(c_1{\varphi}_{3} +c_2 \bar{\varphi}_{3})  = 0.$$
and  exactly these obstructions make $h^{1,1}_{\rm A}$ jumps from 8 to 7 in $(ii.a)$ and from 8 to 6
in $(ii.b)$.

In the end of the section, we want to give the following observation as an application of the formula.
\begin{proposition}\label{proposition last}
Let $X$ be an non-K$\ddot{a}$hler nilpotent complex parallelisable
manifold whose dimension is more than 2, and $\pi:
\mathcal{X}\rightarrow B$ be the versal deformation family of $X$.
Then the number $h_{\rm A}^{1,1}$ will jump in any neighborhood of
$0\in B$.
\end{proposition}
\begin{proof}
From the proof of \cite{ye} propsition 4.2, we know there exists an element
$[\theta]$ in $H^0(X,\Omega_X)$ whose $o_1([\theta]) \neq 0$. It is easy to check that $\theta$
also represents an element in $\HH^2({X,\sB_{2,2}}^{\bullet})$, let us denote it by $[\theta]_{\sB}$ and
it is also easy to check that  $o_1([\theta]) = o_1([\theta]_{\sB})$ in this case. Therefore the number $h_{\rm A}^{1,1}$ will jump in any neighborhood of
$0\in B$.

\end{proof}

\end{document}